 \documentclass[11pt]{amsart}
\usepackage{amsmath,amsthm,amsfonts,amssymb,amsxtra,graphicx}
\usepackage{hyperref}

\usepackage[alphabetic,initials,nobysame]{amsrefs}

\usepackage[all]{xy}
\swapnumbers

\renewcommand{\geq}{\geqslant}
\renewcommand{\leq}{\leqslant}
\swapnumbers
\newcommand{\tildeHd}{\widetilde{\mathscr{H}}_{d}}
\newcommand{\Hdq}{{\mathscr{H}}_{d}(q)}
\newcommand{\tildeHdq}{\widetilde{\mathscr{H}}_{d}(q)}
\newcommand{\bash}{\backslash}
\newcommand{\IGNORE}[1]{}
\newcommand{\whZ}{\widehat{\Zz}}
\newcommand{\gen}{\mathrm{genus}}
\newcommand{\Id}{\mathrm{Id}}
\newcommand{\area}{\mathrm{area}}
\newcommand{\acos}{\mathrm{acos}}

\newcommand{\Pq}{\mathcal{Q}}
\DeclareMathOperator{\Pic}{Pic}
\newcommand{\OO}{\order}

\newcommand{\mat}[4]{\left[\begin{array}{cc}#1 & #2 \\
                                         #3 & #4\end{array}\right]}
\newcommand{\ic}[1]{\mathfrak{#1}}

\DeclareMathOperator{\GL}{GL}

\DeclareMathOperator{\SL}{SL}

\newcommand{\HH}{\mathscr{H}}
\newcommand{\A}{\mathbb{A}}

\DeclareMathOperator{\modu}{\mathrm{mod}}

\newcommand{\tensor} {\otimes}

\newcommand{\beq}{\begin{displaymath}}
\newcommand{\eeq}{\end{displaymath}}
\newcommand{\beqn}{\begin{equation}}
\newcommand{\eeqn}{\end{equation}}

\def\stacksum#1#2{{\stackrel{{\scriptstyle #1}}{{\scriptstyle #2}}}}
\def\peter#1{\langle #1\rangle}
\newcommand{\Perp}{\underline{\mathrm{Perp}}}
\def\ov#1{\overline{#1}}

\newcommand{\Hdstar}{{\tildeHd}^{*}}

\newcommand{\mcB}{\mathcal{B}}
\newcommand{\mcP}{\mathcal{P}}
\newcommand{\mcE}{\mathcal{E}}
\newcommand{\im}{\mathrm{Im}}

\newcommand{\refs}{\eqref}
\newcommand{\ra}{\rightarrow}

\newcommand{\Nn}{{\mathbb{N}}}
\newcommand{\Zz}{{\mathbb{Z}}}
\newcommand{\Rr}{{\mathbb{R}}}
\newcommand{\Qq}{{\mathbb{Q}}}
\newcommand{\Ff}{{\mathbb{F}}}
\newcommand{\Aa}{{\mathbb{A}}}

\newcommand{\mcT}{{\mathcal{T}}}
\newcommand{\mcA}{{\mathcal{A}}}

\newcommand{\dev}{\mathrm{dev}_{d}}

\newcommand{\mfa}{\mathfrak{a}}

\newcommand{\mfq}{\mathfrak{q}}

\newcommand{\mfp}{\mathfrak{p}}
\newcommand{\eps}{\varepsilon}

\numberwithin{equation}{section}
\newcommand{\tr}{\mathrm{tr}}
\newcommand{\Nr}{\mathrm{Nr}}

\newcommand{\Aut}{\mathrm{Aut}}

\newcommand{\Ad}{\mathrm{Ad}}

\newcommand{\disc}{\mathrm{disc}}

\newcommand{\C}{\mathbb{C}}

\newcommand{\ord}{\mathrm{ord}}
\newcommand{\G}{\mathbf{G}}
\newcommand{\B}{\mathrm{B}}
\newcommand{\PB}{\mathrm{PB}^{\times}}

\newcommand{\PGL}{\mathrm{PGL}}
\newcommand{\SO}{\mathrm{SO}}

\newcommand{\order}{\mathscr{O}}

\renewcommand{\H}{\mathscr{H}}

\newcommand{\Z}{\mathbb{Z}}
\newcommand{\Q}{\mathbb{Q}}

\newcommand{\adele}{\mathbb{A}}
\newcommand{\Aaf}{\mathbb{A}_{f}}
\newcommand{\Norm}{\mathrm{N}}

\newcommand{\be}{\begin{equation}}
\newcommand{\ee}{\end{equation}}
\newcommand{\bea}{\begin{eqnarray*}}
\newcommand{\eea}{\end{eqnarray*}}

\newcommand{\bfx}{\mathbf{x}}
\newcommand{\rmr}{\mathrm{red}}
\newcommand{\obfx}{\overline{\mathbf{x}}}
\newcommand{\Hd}{\mathscr{H}_d}

\newcommand{\Gd}{\Hd}
\newcommand{\Bdel}{\mathcal{B}_\delta}
\newcommand{\Bscr}{\mathcal{B}}
\newcommand{\x}{\mathbf{x}}

\renewcommand{\P}{\mathcal{P}}

\newcommand{\Qp}{\Q_p}
\newcommand{\Zp}{\Z_p}
\newcommand{\Zhat}{\whZ}

\DeclareFontFamily{OT1}{rsfs}{}
\DeclareFontShape{OT1}{rsfs}{n}{it}{<-> rsfs10}{}
\DeclareMathAlphabet{\mathscr}{OT1}{rsfs}{n}{it}

\theoremstyle{plain}
\newtheorem*{thm*}{Theorem}
\newtheorem{thm}[subsection]{Theorem}
\newtheorem*{cor}{Corollary}

\newtheorem{lem}[subsubsection]{Lemma}
\newtheorem*{example*}{Example}

\newtheorem{prop}[subsection]{Proposition}
\newtheorem*{prop*}{Proposition}
\newtheorem*{lem*}{Lemma}

\theoremstyle{definition}

\newtheorem{rem}{Remark}[section]

\newtheorem*{rem*}{Remark}
\begin{document}

\newpage
\title[Linnik's ergodic method and points on spheres]{Linnik's ergodic method and\\ the distribution of integer points on spheres}
\author[J. S. Ellenberg, Ph. Michel and A. Venkatesh]{Jordan S. Ellenberg, Philippe Michel and Akshay Venkatesh}
\thanks{
We thank agencies that have generously supported our research. J.E. was partially supported by  NSF-CAREER Grant DMS-0448750 and a Sloan Research Fellowship; A.V.
was partially supported by a Packard fellowship, a Sloan Research Fellowship, 
and an NSF grant; Ph.M. is partially supported by the Advanced research Grant 228304 from the European Research Council
and the SNF grant 200021-125291.}
 \begin{abstract}
 We discuss Linnik's work on the distribution of integral solutions to $$x^2+y^2+z^2 =d,\hbox{ as $d \rightarrow \infty$.}$$
We give an exposition of Linnik's ergodic method; indeed, by
using large-deviation
 results for random walks on expander graphs, we establish 
 a refinement of his equidistribution theorem. We discuss the connection of these ideas with modern developments (ergodic theory on homogeneous spaces, $L$-functions). 

 \end{abstract} 
 \maketitle

 \tableofcontents

\part{Linnik equidistribution theorems}
 \section{Introduction}
 Let $d>1$ be an integer which, for simplicity of exposition, we assume to be {\em squarefree}, and let $\Hd$  be the set of integer points on a $2$-dimensional sphere of radius ${d}^{1/2}$:
$$\Hd:=\{\bfx=(x,y,z)\in\Zz^3,\ x^2 + y^2 + z^2 =d\}.$$

The study of $\Hd$ is a classical question of number theory. Surprisingly,
there are interesting results about $\Hd$ that have been proved in the last
two decades, and simply stated problems that remain unresolved.  In increasing order of fineness, one may ask:
\begin{enumerate}
\item When is $\Hd$ nonempty?
\item If nonempty, how large is $\Hd$, and how can we generate points in $\Hd$?
\item If $\Hd$ get large,  how is it distributed on the sphere of radius $d^{1/2}$?
\end{enumerate}

The first question was studied by Legendre and the answer is: 
$$\hbox{\em $\Hd$ is nonempty if, and only if, $d$ is not of the form $4^a(8b-1)$ ($a,b\in\Nn$),}$$
(or put in other terms, the quadratic equation  $x^2+y^2+z^2=d$ satisfies the {\em Hasse principle}.)
Legendre's proof from 1798 however was incomplete\footnote{Legendre assumed the existence of primes in arithmetic progressions, which was proven about 40 years later by Dirichlet.} and the first complete proof was given by Gauss three years later in his {\em Disquisitiones arithmeticae} \cite{Gauss}.

The second question is somewhat subtler and its resolution is the consequence of the work of several people.
 The fundamental insight however come from Gauss work who showed that $\Hd$ is closely connected with
   the {\em set of classes of binary quadratic forms of discriminant $-d$}. 
 
This relation amounts, in more modern terms, to the existence of a natural {\em action} of the ideal class group of the quadratic ring $\Zz[\sqrt{-d}]$ on the quotient $\SO_3(\Z)\bash\Hd$. This action is, in fact, {\em transitive}  (at least if $d$ is squarefree, which we assume here) and is faithful if and only if $d \equiv 3$ modulo $8$. An exposition of these facts is given in \S \ref{torsorA}. In particular, whereas $\SO_3(\Z)\bash\Hd$ does itself not have a natural group structure (what would the identity be?) the notion of ``arithmetic progression'' makes sense on $\SO_3(\Z)\bash\Hd$. 
 
 A first consequence of the existence of this action is an exact formula relating $|\Hd|$
to the class number  of $\Qq(\sqrt{-d})$.  From there, Dirichlet's class number formula expresses
$|\Hd|$ in terms of the residue at $1$ of the Dedekind $\zeta$-function of that field, and {\em Siegel's theorem} controls this value quite precisely, yielding
\begin{equation}\label{Hdsize}|\Hd|=d^{1/2+o(1)}.
\end{equation}

The third question is the main focus of this paper; whereas it is not at first clear that is a worthy successor to the first two, its investigation has proved very rich. 
Progress on it  has been  entwined with the study of modular $L$-functions,  as well as to the study of dynamics on homogeneous spaces. The first significant answer regarding this question are due to Y. V. Linnik who, in the late 50's, proved amongst other results the following,
 
 \begin{thm}[Linnik] \label{linnikinfty} As $d\ra+\infty$ amongst the squarefree integers satisfying $d\equiv \pm 1(5)$, the set
$$\{\frac{\mathbf{x}}{\sqrt{d}},\mathbf{x} \in \Hd \}\subset S^2$$
becomes equidistributed on the unit sphere $S^2$ with respect to the Lebesgue probability measure.
\end{thm}

  \begin{figure}
\centering
\resizebox{!}{6.5cm}{
\includegraphics{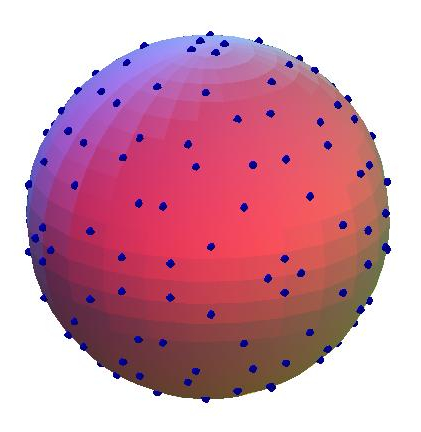}
}

\centering
\resizebox{!}{6.5cm}{
\includegraphics{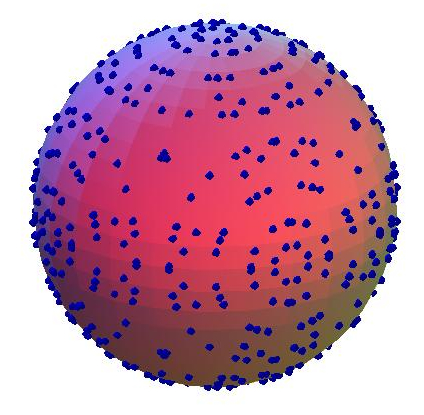}
}

\centering
\resizebox{!}{6.5cm}{
\includegraphics{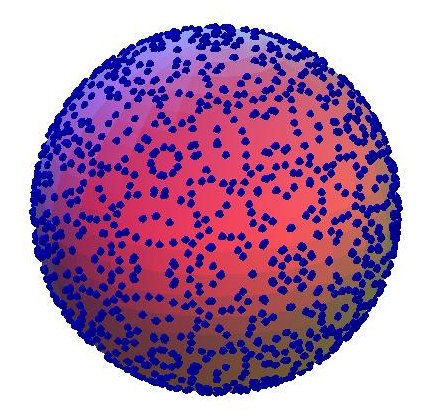}
}
\caption{$d^{-1/2}\Hd\subset S^2$ for $d=101,\ 8011,\ 104851$}
\end{figure}   
In explicit terms, that means that if 
$$\rmr_{\infty}: \Hd\mapsto S^2$$
 denotes the ``scaling'' map $\bfx\mapsto d^{-1/2}.\bfx$, then
for any measurable subset $\Omega \subset S^2$ whose boundary has Lebesgue measure zero
$$\frac{|\rmr_{\infty}^{-1}(\Omega)|}{ |\Hd|}= {\area(\Omega)}(1+o(1)),\ d\ra+\infty;$$
(we take the normalization $\area(S^2)=1$).

Linnik obtained this by an ingenious technique which he called the ``ergodic method.''  This method was generalized later (notably by Linnik's student, Skubenko) to establish several remarkable results about the distribution of the representations of 
large integers by (integral) ternary quadratic forms \cite{EPNF}. Until recently, Linnik's ergodic
   method remained surprisingly little-known, although simplified treatments were
   given by a number of authors  \cite{Teterin,Malyshev}; one possible reason is that the
   method did not fit into ergodic theory as the term is now usually understood, i.e. dynamics of a measure-preserving transformation.  
   
   Removing the constraint $d \equiv \pm 1$ modulo $5$ proved to be very difficult, and 
   was resolved only thirty years later by W. Duke \cite{Duke}; see \S \ref{dukeHA} for discussion.

    The aim of the present paper is to revisit and explain in a slightly different language Linnik's original approach   and to present further refinements which do not
     seem present in Linnik's work and do not seem accessible to the other known methods. 
     We have endeavoured to simplify the original proofs of Linnik as far as possible, and have favoured explicitness and directness over generality and (more regrettably) over more conceptual approaches. 

\subsection{A refinement of Theorem \ref{linnikinfty}} 

Consider, for $x\in S^2$ and $\rho>0$, the spherical cap $\Omega(x,\rho)$ of center $x$ and radius $\rho>0$
 (i.e. the ball  of center $x$ and radius $\rho$ with respect to the standard Riemannian metric on $S^2$).
Define the {\em deviation} of $\Omega$ to be $$\dev(\Omega) = \frac{1}{\area(\Omega)}\frac{|\rmr_{\infty}^{-1}(\Omega)|} { |\Hd|  }-1.$$  Theorem \ref{linnikinfty} is equivalent to the fact that for any $x\in S^2$, $\rho>0$, 
$$\dev(\Omega(x,\rho))\ra 0\hbox{ as }d\ra\infty.$$
We shall establish the following refinement of the prior result:
\begin{thm*} \label{MVinfty} 
Fix $\delta , \eta > 0$ . For any $\rho\geq d^{-1/4+\delta}$,
the measure of $x \in S^2$ for which $|\dev(\Omega(x,\rho))| \geq \eta$,
tends to $0$, as $d \rightarrow \infty$ as above (i.e. squarefree and $\equiv\pm 1(5)$). 
\end{thm*}

Note that this theorem implies Theorem \ref{linnikinfty}. Moreover, in view of \refs{Hdsize}, it is
easy to see that the exponent $1/4$ in the above statement  is {\em optimal}. It would, of course,
be desirable to replace ``{\em tends to $0$ as $d\rightarrow \infty$}'' by
``{\em equals $0$ for sufficiently large $d$.}'' This seems extremely difficult. On the other hand, by different methods, it has been shown that this holds in a restricted range $\rho\geq d^{-\kappa}$ for $\kappa>0$ some absolute constant
(cf. \S \ref{harmonic}).

\subsection{A modular analog} It is also possible to prove versions of the above results where, instead of studying the 
position of $\Hd$ on the sphere, we study the congruence properties of points in $\Hd$. 
For $q$ an integer coprime with $d$, let $\Hd(q)$ denote the ``sphere modulo $q$''
$$\Hd(q):=\{\obfx=(\ov x,\ov y,\ov z)\in(\Zz/q\Zz)^3,\ \ov x^2 + \ov y^2 + \ov z^2 \equiv d \modu q\}.$$

\begin{thm*}[Linnik] \label{linnikQ}  Let $q$ be a fixed integer, coprime with $30$. As $d\ra+\infty$ amongst the squarefree integers satisfying $d\equiv \pm 1(5)$, $(d,q)=1$, the multiset
$$\{ \x\ (\modu q),  \ \x \in \Hd\} \subset \Hd(q)$$
becomes equidistributed on $\Hd(q)$ with respect to the uniform measure. 
\end{thm*} 

In explicit terms, that means that if $$\rmr_q:\Hd\mapsto\Hd(q)$$ denotes the reduction modulo $q$ map, then,
for any $\obfx\in\Hd(q)$,
$$\frac{|\rmr_{q}^{-1}(\obfx)|}{ |\Hd| }= \frac{ 1 }{|\Hd(q)|}(1+o(1)),\ \mathrm{as}\ d\ra+\infty.$$
We prove also the following refinement: define the {\em deviation} at $\obfx\in\Hd(q)$ to be
$$\dev(\obfx) = \frac{|\rmr_q^{-1}(\obfx)|}{|\Hd|/|\Hd(q)|}- 1$$
\begin{thm} \label{mvQ} 
Fix $\nu, \delta > 0$ and suppose that $q \leq d^{1/4-\nu}$ and $(q,30) = 1$.
The fraction of $\obfx \in \Hd(q)$ for which $|\dev(\obfx)| > \delta$ tends to zero as $d \rightarrow \infty$ with $d \equiv \pm 1 (5)$.
\end{thm}

A qualitative consequence of this Theorem is that: 

\noindent {\em As long as
$q \leq d^{1/4-\delta}$, { almost any solution} to $\bar{x}^2 + \bar{y}^2 + \bar{z}^2 = d$
with $(\bar{x}, \bar{y}, \bar{z}) \in (\Z/q\Z)^3$, can be lifted to a solution $x^2+y^2+z^2 =d$, $(x,y,z)\in\Zz^3$}.

 Again, it is natural to surmise that this is true for {\em all} solutions; but this appears to be a very difficult problem. It is also interesting to consider numerics related to this issue.  
 
In this paper, we shall discuss the proof of Theorem \ref{mvQ}  in some detail; 
all the other results may be obtained by modifying that proof; see \S \ref{Extensions} for a discussion of how to carry this out.

\subsection{Plan of the paper}
In \S \ref{mvQproof} we describe the main steps of the proof of Theorems \ref{linnikinfty} and \ref{mvQ}; for simplicity we focus on the second theorem and briefly indicate the modifications necessary to handle the first. Each of these steps is then explained in detail in the remainder of the paper:
\begin{itemize}
  \item[--] Part 1 concerns those parts of the proof that are most conveniently presented in classical language:
  \begin{itemize}  
  \item[-] In \S \ref{torsorA}, we describe the natural action of the class group on the quotient $\SO_{3}(\Zz)\bash\Hd$ and make it rather explicit. For that purpose, it is particularly useful to express everything in terms of quaternions. Presentations of similar material can be found in \cite{Venkov,Venkov2} and \cite{Shemanske}; our approach is slightly different. 
  \item[-] In \S \ref{linniklemma}, we establish the important ``basic lemma'' of Linnik.
  \end{itemize}
  \item[--] Part 2  -- comprised of
   \S \ref{sec:adelicHdq},  \S \ref{sec:adelicHd},  and \S \ref{sec:adelicredq} 
   is concerned with transferring some of the key objects -- $\Hd, \Hd(q)$ --
  in terms of adelic quotients. This will be a key tool in proofs. 
  \item[--] Part 3 is concerned with the part of the proof that is related to expanders. 
  \begin{itemize}
  \item[-] In \S \ref{sec:Hdqexpander} we prove that $\Hd(q)$ -- endowed with a graph structure that we shall describe -- is an expander graph. 
  \item[-] In \S \ref{expander}, we recall some basic facts from the theory of random walks on expander graphs. In particular we give a self-contained proof of a {\em large deviation estimate} (for non-backtracking paths) on such graphs.
  \end{itemize}
  \item[--] Part 4 is concerned with miscellaneous extensions of the main topics. 
  \begin{itemize}
  \item[-] In \S \ref{Extensions}, we discuss various possible extensions of Linnik's ergodic method.
  \item[-] In \S \ref{harmonic}, we recall the alternative approach of Duke and others to these equidistribution problems, which is based on harmonic analysis and more precisely on the theory of automorphic forms. We also compare the two approaches.  
  \end{itemize}
\end{itemize}

\subsection{Acknowledgements}  
The ideas of this paper are based on those of Linnik, and, indeed, this paper should be
regarded in considerable part as an exposition of his work.
This paper also draws on the ideas in the work of the latter two authors
with M. Einsiedler and E. Lindenstrauss, as well as work of the first- and last- named author \cites{ELMV1,ELMV3,ELMV2,EV}. 
The novel results of the paper were obtained in 2005; we apologize for the delay in bringing them to print.

We also thank P. Sarnak for encouragement of the project, R. Masri for reading an early version and J.-P. Serre for his comments and criticism; finally special thanks are due to G. Harcos who read very carefully the whole manuscript and made numerous corrections and comments on that occasion.

\section{An overview of the ergodic method}\label{mvQproof}

We now present an overview of the proof of Theorem \ref{mvQ}.  We will try to isolate the main steps of the proof, each one of which contains some key results of a more general mathematical interest.  In the present section, we treat each of these results as a black box; in the latter part of the paper, we ``open the black boxes'' one by one and provide complete proofs.

\subsection{Assumptions and notation.}  \label{notn}
Throughout the paper we shall make the following assumptions: 
\begin{enumerate}
\item $d$ will always
denote a squarefree integer, not congruent to $7$ modulo $8$, and congruent to $\pm 1$ modulo $5$. 
\item $q$ will always denote an integer prime to $30$. 
\end{enumerate}

\begin{rem*} One could replace $5$ by an arbitrary fixed prime $p$ and the condition $d \equiv \pm 1$ modulo $5$
 by the assertion that {\em $p$ is split in $\Q(\sqrt{-d})$}. The assumption that $q$ is prime to $30$ is for convenience and could be removed entirely.  
 \end{rem*}
 
 Write $K=\Qq(\sqrt{-d})$; we denote by $\OO_{K}$ the ring of integers of $K$, and by $\Pic(\OO_{K})$ the ideal class group of $\OO_K$. We also fix a square root of $-d$ in $K$,  and denote it by $\sqrt{-d}$ without further commentary.

\subsection{Some natural quotients} The problems considered in the introduction admit ``obvious'' {\em symmetries}
owing to the evident action of $\SO_{3}(\Zz)$ on $\Hd$, $S^2$ or $\Hd(q)$. We denote the corresponding quotients by $$\tildeHd = \SO_3(\Zz)\bash \Hd,\ \widetilde{\Hd}(q)= \SO_3(\Zz)\bash \Hd(q),\  \widetilde{S^2}= \SO_3(\Zz)\bash S^2.$$ and denote by $\tilde\bfx$ the orbit $\SO_{3}(\Zz)\bfx$ of any element in the above sets.
 
  For the purposes of our main theorems, there is no essential difference between working with $\Hd$, $S^2$ or $\Hd(q)$  and working with $\tildeHd$, $\widetilde{S^2}$ or $\widetilde{\Hd}(q)$ (since $\SO_{3}(\Zz)$ is finite). It is often conceptually clearer to consider the question of how $\tildeHd$ becomes distributed in  $\widetilde{\Hd}(q)$ or 
$\widetilde{S^2}$ than the similar question for $\Hd$ and $S^2$ or $\Hd(q)$;
this becomes clear when considering these problems for other ternary quadratic forms, especially indefinite forms.
 However, the latter formulation being slightly more classical, we will make the (slight) extra effort required to state theorems on $\Hd$.

Recall that if $G$ is a group, a {\em homogeneous space} for $G$ is simply a set $X$ on which $G$ acts transitively (i.e. for any $x,x'\in X$, there is $g\in G$ such that $g.x=x'$). In such a case, the stabilizers of the elements of $X$ under this action are all conjugate.

\begin{prop}[$\tildeHd$ is a $\Pic(\OO_{K})$-homogeneous space]\label{torsor0}  Let $d>3$ be a squarefree integer not congruent to $7$ mod $8$. Then there exists
a natural action of $\Pic(\OO_K)$ on $\tildeHd$, making $\tildeHd$ into a homogeneous space for $\Pic(\OO_K)$. The stabilizer of any point is trivial if $d \equiv 3$ modulo $4$,
and the order $2$ subgroup generated by a prime above $2$ if $d\equiv 1,2$  modulo $4$.
 In particular,
$$|\Hd|=24|\Pic(\OO_{K})|\hbox{ when $d\equiv 3\ (8)$ }$$
and 
$$|\Hd| = 12|\Pic(\OO_{K})|\hbox{ when $d\equiv 1,2\ (4)$.}$$

\end{prop}

We prove a slightly more precise version of this statement in \S \ref{quaternions} and explain its adelic manifestation in \S \ref{classgroupadelicsection}. 

Given $\mfa\subset \OO_{K}$ an ideal, we denote by $[\mfa].\tilde\bfx$ the action of its corresponding ideal class on some element $\tilde\bfx\in\tildeHd$.

\subsection{The $\Pic(\OO_{K})$-action} \label{explication5}
While it is possible to describe explicitly the action of $\Pic(\OO_{K})$ on $\tildeHd$ (or at least the action of
the prime ideal classes),  for the proof of the main results we will only need the action of the primes above $5$.
Since $d\equiv \pm 1(5)$, the prime $5$ splits in $K$, which is to say that the principal ideal $5\OO_{K}$ factors  into a product of two prime ideals
$$5\OO_{K}=\mfp.\mfp'.$$ 
We shall now realize explicitly the action of $[\mfp]$ and the group it generates.

Let $A, B, C$ be, respectively, rotations by angles $\acos(-4/5)$ around the $x,y,z$ axes.
They are described by the matrices:

$$ A =  \frac{1}{5} \left( \begin{array}{ccc} 5 & 0 & 0  \\ 0 & -4 & 3 \\ 0 & -3 & -4  \end{array} \right),
B =\frac{1}{5} \left( \begin{array}{ccc} -4 & 0 & 3  \\  0 & 5 & 0 \\  -3 & 0 &-4  \end{array} \right),
C = \frac{1}{5}\left( \begin{array}{ccc}  -4 & -3 & 0 \\ 3 & -4 & 0 \\ 0 & 0 & 5 \end{array} \right).
 $$

\begin{prop}[The action of a prime ideal]  \label{lem:Linnik1} For $\x \in \Hd$, exactly two of
 $$\{A \x, A^{-1} \x, B \x, B^{-1} \x, C \x, C^{-1} \x\}$$ belong 
to $\Hd$.  The classes of those two points in $\tildeHd$ are
$[\mathfrak{p}] . \tilde\x$ and $[\mathfrak{p}']. \tilde\x=[\mathfrak{p}]^{-1} . \tilde\x$. 
\end{prop}
\proof Since multiplication of an element of $\Hd$ by either of the matrices $A,B,C$ or their inverse produce vector with rational coordinates with denominator equal to $1$ or $5$, the first part of the proposition can be verified by direct computation on the set $\Hd(5)$, which is to say, the set of solutions to $x^2 + y^2 + z^2 = d$ in $(\Z/5\Z)^3$; one just has to check that for each $\bar{\bfx}$ in $\Hd(5)$, there are exactly two choices of $M \in \{5A,5A^{-1},5B,5B^{-1},5C,5C^{-1}\}$ satisfying $M\obfx = 0$. The second part will be proved in \S \ref{ss:explicitaction}. 
\qed

Note that Proposition \ref{lem:Linnik1} in fact defines a {\em lifting}
of the action of $[\mathfrak{p}]^{\Zz}$ from $\tildeHd$  to $\Hd$.  This lifting
is, in fact, rather less canonical than the $\Pic(\OO_K)$-action on $\tildeHd$. 

\subsection{Trajectories}\label{primeaction}
We can use Proposition \ref{lem:Linnik1} to determine distinguished trajectories in $\Hd$.  
Write  $$\mcA_{5}:=\{A, A^{-1}, B, B^{-1}, C, C^{-1}\}.$$
To start with, pick any $\x\in\Hd$; now by  Lemma \ref{lem:Linnik1}, there are precisely two matrices in $\mcA_{5}$ -- say $w_{1},\ w_{0}$   -- such that $w_1 \bfx$ and $w_0^{-1} \bfx$ both belong to $\Hd$.  Denote $w_1 \bfx$ by $\bfx_1$, and $w_0^{-1}\bfx$ by $\bfx_{-1}$.  Similarly, there is a unique choice of matrix $w_2 \in \mcA_5 $ such that $w_2 \neq (w_1)^{-1}$ and $w_2 \bfx_1$ belongs to $\Hd$; we denote $w_2 \bfx_1$ by $\bfx_2$.  In this way, repeated application of Lemma~\ref{lem:Linnik1} gives rise to a sequence $(\x_{i})_{i\in\Zz}$ in $\Hd$ such that
$\x_0 = \x$ and for any $i$, $$\x_{i}\in\Hd,\ \x_{i+1} \in \{A \x_i, A^{-1} \x_i, B \x_i, B^{-1} \x_i, C \x_i, C^{-1} \x_i\}-\{\x_{i-1}\}.$$
Alternatively, this string can be represented by the data of $\bfx=\bfx_{0}$ and an infinite (necessarily periodic) word $W_{\x}=(w_{i})_{i\in\Zz},$ in the alphabet $\mcA_{5}$,  satisfying, for any $i$,
\begin{equation}\label{pathproperty}
w_{i+1}\not=(w_{i})^{-1},\ \x_{i}=w_{i}\x_{i-1} \in \Hd.
\end{equation}
A word satisfying the condition $w_{i+1}\not=(w_{i})^{-1}$ is called {\em reduced}; it is easy to see that $W_{\x}$ is the unique reduced word satisfying \eqref{pathproperty}, up to the ``switch directions'' transformation given by switching $w_i$ and $w_{1-i}^{-1}$, or, equivalently, switching $\x_i$ and $\x_{-i}$.  We refer to the sequence $(\x_i)_{i \in \Z}$ as the {\em trajectory} of $\x$.

The equivalence of this trajectory to the one defined by the action of $[\mfp]^\Z$ is explained in \S \ref{ss:picokacts}.

\subsection{An example}

For the sake of concreteness, we include an explicit example.  Take $d=101$.  In this case, $|\Hd| = 168$ and $|\Pic(\OO_{K})|=14$.
The points of $|\Hd|$, up to the action of $\SO_3(\Z)$, are:
$$(10, 1,0), \pm (9, 4, 2), \pm (8, 6,1), \pm (7, 6,4)$$
The {\em trajectory} containing $(10,1,0)$ is:
\begin{gather*} (10,1,0) \stackrel{B}{\rightarrow} (-8,1,-6) \stackrel{C^{-1}}{\rightarrow} (7,4,-6) \stackrel{B^{-1}}{\rightarrow} 
(-2, 4,9) \stackrel{C^{-1}}{\rightarrow}  \\
 (4,-2,9) \stackrel{A}{\rightarrow} (4, 7,-6) \stackrel{C^{-1}}{\rightarrow} 
(1,-8,-6) \stackrel{A^{-1}}{\rightarrow} (1,10,0) \ldots
 \end{gather*}
$$W_{(10,1,0)}=\dots {B}^\star C^{-1}B^{-1}C^{-1}AC^{-1}A^{-1}\dots\\
$$
with ${B}^\star=w_{1}$.

 We note that after seven steps (and thus, after any multiple of seven steps) the trajectory returns to the $\SO_3(\Z)$-orbit of $(10,1,0)$.  This periodicity of order $7$ reflects the fact that the class of
  a prime ideal above $5$ has order $7$ in the class group of $\Q(\sqrt{-101})$, which is cyclic of order $14$. 
 \label{example}

\subsection{The lengths of the trajectories} We should emphasize that it is certainly possible, and in some sense probably typical, for $[\mfp]^\Z$ to be all or most of $\Pic(\OO_K)$.

What we know in this direction is rather minimal. On the one hand, it is easy to see that  $$|[\mfp]^{\Zz}|\gg \log(d);$$ while this lower bound goes to $\infty$ with $d$, it remains quite small compared with the size of $|\Pic(\OO_{K})|=d^{1/2+o(1)}$. 
 As far as we know, it is unknown whether there exist infinitely many squarefree $d\equiv 1(20)$ such that $|[\mfp]^{\Zz}|$ is greater than some positive power of $d$.
 
These considerations are, of course, related to the question:

{\em for which $d$ does the process described in Proposition~\ref{lem:Linnik1}\linebreak traverse all points in $\tildeHd$?}

 \noindent This condition on $d$ turns out to be equivalent to the condition that the prime ideal $\mathfrak{p}$ above $5$ and the $2$-torsion ideal class above $2$ generate the ideal class group $\Q(\sqrt{-d})$.  The question of {\em whether this happens for infinitely many $d$} is an interesting and seemingly difficult one; heuristically it is reasonable to suppose
that this happens a positive fraction of the time.  It is analogous to Gauss's famous question: 

{\em do there exist infinitely many $d > 0$ for which $\Q(\sqrt{d})$ has class number $1$?}

\noindent The prime above $5$ plays the role in our situation that a real place does in Gauss's problem;  the {\em period} of the process above is analogous to the regulator of a real quadratic field.

  \begin{figure}
\centering
\resizebox{!}{5.5cm}{
\includegraphics{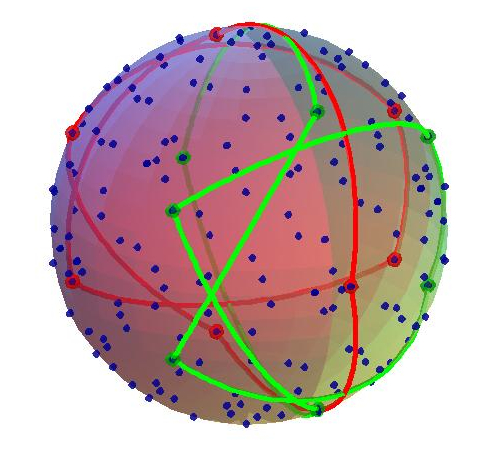}
}
\caption{The trajectories of $(10,1,0)$ (red) and $(6,1,8)$ (green) within $\Hd$}
\end{figure}

\subsection{Trajectories on the graph $\Hd(q)$} \label{HdqExpanderGraph}
We shall now begin to examine solutions to $x^2+y^2+z^2=d$ modulo a
positive integer $q$, as in Theorem \ref{linnikQ}. We shall suppose, as in that theorem,
that $q$ is relatively prime to $30$. 

Under this assumption, the matrices $\mcA_{5}=\{A, A^{-1}, B, B^{-1}, C, C^{-1}\}$
define elements in $\SO_3(\Z/q\Z)$ and act on $\Hd(q)$;  this endows $\Hd(q)$ with a structure of a $6$-regular (undirected)
\IGNORE{\footnote{``undirected'' means precisely that in this construction, we consider as the same
  the edge the link between (say) $v$ to $v'=A.v$ and the link between $v'$ to $v=A^{-1}v'$.}}
   graph  by joining each $ \ov \bfx\in\Hd(q)$ to $$A \ov  \bfx,\ A^{-1}\ov   \bfx,\ B\ov   \bfx ,\ B^{-1} \ov  \bfx,\ C  \ov \bfx,\ C^{-1}  \ov \bfx.$$
     By an abuse of notation we shall also refer to this graph as $\Gd(q).$  More precisely, it is a {\em multigraph} -- we allow multiple edges between pairs of vertices, and edges joining a vertex to itself.

  Now to each point $\x \in \Hd$, we may associate a well-defined {\em marked path}
  on the graph $\Hd(q)$ in the following way. Given $\x \in \Hd$, we constructed in  \S \ref{primeaction} a sequence $(\x_i)_{i \in \Z}$ of points in $\Hd$, the {\em trajectory} of $\x$, which is well-defined up to the substitution $\x_i \mapsto \x_{-i}$.  We denote by $\gamma_{\bfx}$ the reduction of this trajectory to $\Hd(q)$; the data of $\gamma_{\bfx}$ is the sequence of vertices $(\obfx_{i})_{i\in\Zz}$ of $\Hd(q)$ ($\obfx_{i}=\rmr_{q}(\bfx_{i})$), together with a marked basepoint $\obfx_0$ and a choice, for each $i$, of an edge joining $\obfx_i$ to $\obfx_{i+1}$. We denote by $\gamma_\bfx^{(\ell)}$
 the segment of this path of length $2 \ell$, centered at the marked vertex. In other words, 
 $$\gamma_{\bfx}^{\ell}  = (\x_{-\ell}, \x_{-\ell+1}, \dots, \x_{\ell-1}, \x_{\ell}).$$ We may refer to it as the {\em truncated}
  trajectory of length $2\ell$.

 Note that the trajectories arising as $\gamma_\bfx$ are not completely arbitrary paths on $\Hd(q)$; the condition $w_{i+1} \neq w_i^{-1}$ from \refs{pathproperty} implies that $\gamma_\bfx$ never traverses the same edge twice in succession.  (This does not, of course, forbid $\gamma_\bfx$ from traveling from $\ov x$ to $\ov y$ and then back to $\ov x$; it just has to use two distinct edges joining $\ov x$ and $\ov y$.)  A {\em non-backtracking path} in $\Hd(q)$ is a (marked) path which never traverses the same edge twice in succession; equivalently, it can be defined by the data of a marked point $\bar{x}_0$, and a word $W_{\x}=(w_{i})_{i\in\Zz}$ satisfying $w_{i+1} \neq w_i^{-1}$; the vertices $\bar{x}_i$ can be determined inductively by the rule that $\ov x_{i+1}$ is the vertex arrived at by following the edge labeled $w_{i+1}$ from $\ov x_i$.

 For instance, if $d = 101, q=7, P= (7,4,-6)$, we have (see \S \ref{example}) 
\begin{gather*}
 \gamma^{(2)}_P \ : \cdots\ [(3,1,0)] \stackrel{B}\rightarrow [(-1,1,1) ]  \stackrel{C^{-1}}\rightarrow
  [(0,4,1)]^\star \\ \stackrel{B^{-1}}\rightarrow 
  [(-2,4,2)]
   \stackrel{C^{-1}}\rightarrow  [(4,-2,2)] \cdots
\end{gather*}
 Here $\star$ denotes the marked vertex, and we have used square brackets $[\dots]$ to denote
 reduction modulo $7$.

\IGNORE{ We have constructed on the graph $\Hd$ a collection of non-backtracking marked paths,
 the $q$-trajectories associated with the $\bfx\in\Hd$. These are $d^{1/2+o(1)}$ periodic paths whose periods are  $\gg \log d$.
  We will show that these paths behave like random marked non-backtracking walks as $d\ra+\infty$;
   that is walks on $\Hd$ obtained by choosing randomly an origin $\bfx_{0}$ then choosing randomly 
   an edge $w^+_{i}\in \mcA_{5}$ originating from this vertice; this determine a new vertice $\x_{1}=w^+_{1}\x_{0}$
   and we now choose $w^+_{2}$ randomly in $\mcA_{5}-\{(w^+_{1})^{-1}\}$ by iterating the above process we obtain
    a random forward non-backtracking walk; likewise, we need to repeat this operation to define
    the backward walk, by choosing randomly an edge $w^-_{1}\in \mcA_{5}-\{w^+_{1}\}$
     which determines $\x_{-1}=w^-_{1}\x_{0}$ etc...}

\begin{prop}[Shadowing lemma] \label{Naddles} The sequences $W_{\mathbf{x}} $ and
$W_{\mathbf{x}'}$ coincide for $-\ell+1 \leq i \leq \ell$ 
if and only if $\x \equiv \pm \x'$ modulo
$5^{\ell}$.  In particular, 
$\gamma_\x^{(\ell)} = \gamma_{\x'}^{(\ell)}$ only if $\x \equiv \pm \x'$ modulo $q 5^{\ell}$. \end{prop} 
This proposition states roughly that if two points have their trajectory agree for a long time then their initial points have to be $5$-adically close as well as congruent modulo $q$. 
 (  In particular, the trajectories  $\gamma_{\x}$, $\gamma_{\x'}$ are equal if and only if $\x=\pm\x'$.)

The ``if'' part is easy; the ''only if'' is not much harder, although this is less apparent from our definitions.  Proposition \ref{Naddles} is proved in \S \ref{Naddlesproof}, using geometric properties of the Bruhat-Tits building of   $\SO_{3}(\Qq_{5})\simeq\PGL_{2}(\Qq_{5})$.

Our aim is to show that the integral points $\Hd$
and their associated trajectories $\{\gamma_{\x}, \x\in\Hd\}$ are well distributed, in some sense, on  the graph $\Hd(q)$. (In interpreting this statement, one should imagine that $q$ is fixed, or not increasing too quickly, whereas $d \rightarrow \infty$.)   To obtain this, 
we shall use Proposition \ref{Naddles} in conjunction with:

\begin{prop}[Linnik's basic Lemma] \label{BasicLemma}
Let $e\in\Zz$ such that $|e|< d$. The number of pairs $(\x_1, \x_2) \in \Hd^2$ with dot product $\x_1 \cdot \x_2 = e$
 is $\ll_{\varepsilon} d^{\varepsilon}$. 
\end{prop}

This is proved in \S \ref{linniklemma}. It corresponds to the ``basic lemma'' in Linnik's ergodic method and is in a sense, a
 generalization of the well known bounds
$$\tau(d)=\sum_{ab=d}1\ll_{\eps}d^\eps,\ r_{2}(d)=\sum_{a^2+b^2=d}1\ll_{\eps}d^\eps.$$

Indeed, bounding the divisor function $\tau(d)$, or the number of representations $r_2(d)$ of $d$ as the sum
 of two squares, amount to bounding the number of representations of the rank $1$ form $dx^2$
  by the binary quadratic forms $Q(x,y)=xy$ and $Q(x,y)=x^2+y^2$ respectively.
   By comparison, Proposition \ref{BasicLemma} concerns the number of ways to represent the rank two form $d x^2 + e x y + d y^2$
    by the rank three form $x^2+y^2+z^2$.

\begin{prop}\label{meansquare} Let $\Sigma(d,\ell,q)$ be the number of pairs $(\x_1, \x_2) \in \Hd^2$ with $\gamma_{\x}^{(\ell)}=\gamma_{\x'}^{(\ell)}$.  Then 
$$\Sigma(d,\ell,q) \ll_{\eps} |\Hd|+d^\eps\left(1+\frac{d}{q^25^{2\ell}}\right),
$$
for any $\eps>0$.
\end{prop} 
\proof Indeed,  $\gamma_{\x}^{(\ell)} = \gamma_{\x'}^{(\ell)}$ implies that $\x\pm \x'\equiv 0\ (q5^\ell)$
for some choice of $\pm$; hence either $(\x + \x').(\x + \x')\equiv 0\ (q^25^{2\ell})$ or $(\x - \x').(\x - \x')\equiv 0\ (q^25^{2\ell})$ Since $\x.\x=\x'.\x'=d$, we have
$$2 \x . \x' \equiv \pm 2 d\hbox{ mod }q^2 5^{2 \ell}$$
in either case.  
Thus the number of pairs $(\x,\x')$ such that $\gamma_{\x}^{(\ell)} = \gamma_{\x'}^{(\ell)}$ is bounded by
$$|\Hd|+\sum_\stacksum{|e|<d}{e\equiv\pm d(q^2 5^{2 \ell})}|\{(\x,\x')\in\Hd^2,\ \x.\x'=e\}|\ll_{\eps} |\Hd|+d^\eps\left(1+\frac{d}{q^25^{2\ell}}\right).
$$
\qed

\begin{prop}[$\Hd(q)$ is an expander] \label{JLD} For any $q$ coprime with $30$, the graph $\Gd(q)$ is connected and non-bipartite. In particular, its adjacency matrix, 
$A(\Gd(q))$, has a {\em spectral gap} and more precisely, if $$\lambda_1=6\geq\lambda_2\geq\dots \lambda_{|\Hd(q)|}\geq -6$$
 denote the eigenvalues of $A(\Gd(q))$, then
 $$|\lambda_j| \leq 2 \sqrt{5}, \ \ j \neq 1.$$
\end{prop}
This is
 proved in Section \S \ref{JLDproof} using the adelic description of $\Hd(q)$. Indeed, the graphs $\Hd(q)$ are closely related to the original Ramanujan graphs of Lubotzky-Phillips-Sarnak\cite{LPS}.
 The content of this assertion is equivalent to the (optimal) Ramanujan
  bound on the Fourier coefficients of weight $2$ holomorphic forms of level up to $18.q^2$.  Thus, the existence of a spectral gap follows from any nontrivial bound for these Fourier coefficients -- for instance, the works of Kloosterman or Rankin.

The bound on the spectral gap in Proposition~\ref{JLD} says that the family of $6$-valent graphs $\bigl(\Gd(q)\bigr)_{(q,10)=1}$
 form a family of $\frac{3-\sqrt 5}{3}$ {\em expander} (in fact Ramanujan) graphs as $q \rightarrow \infty$.  
 We refer to \cite{Lubot} for an extensive and motivated discussion of expander graphs, their properties,
  applications, and their construction via automorphic forms. 

 We shall use the fact that  $\Hd(q)$ is an expander through: 
 \begin{prop}[Large deviation estimates]\label{chernoff} 
Fix $\eta, \eps > 0$. For any subset $ \Bscr\subset \Hd(q)$ with $|\Bscr| \geq \eta \Hd(q)$, the fraction of non-backtracking  paths $\gamma$
of length $2\ell$
satisfying:
$$\biggl| \frac{|\gamma \cap \Bscr|}{2\ell+1} - \frac{|\Bscr|}{|\Hd(q)|} \biggr|  \geq \eps$$
is bounded by $c_1 \exp(-c_2 \ell)$, where $c_1, c_2$ depend only on $\eps,\eta$. 
\end{prop}
By $|\gamma \cap \Bscr|$ we mean the number of $i \in [-\ell,\ldots,\ell]$ such that the vertex $\obfx_i$ of $\gamma$ is contained in $\Bscr$; in other words, the ``amount of time'' $\gamma$ spends in $\Bscr$ if one imagines moving along the trajectory at a constant speed.  It is then natural to compare the portion of time  spent by a path in $\Bscr$ (i.e. the ratio $|\gamma \cap \Bscr|/(2\ell+1)$) with the probability of being in $\Bscr$ (i.e. ${|\Bscr|}/{|\Hd(q)|}$). Large deviation estimates show that with high probability, long paths spend the right amount of time in any large enough subset $|\Bscr|$.

Such estimates, first proved by Chernoff for complete 
 graphs, are a well-known and useful tool in different contexts; for instance it has been fruitfully applied in computer science
  (see \cite{HLW}).   
  We were unable to find a reference in the existing literature for the particular version we need here
   (a large deviation estimate for non-backtracking walks), and so give a proof from first principles in \S \ref{arcgraph} (Proposition \ref{chernofflemma}).

\subsection{Conclusion of the proof}

We conclude this section by explaining how
Propositions \ref{Naddles},
 \ref{BasicLemma} and \ref{chernoff} together imply Theorem \ref{mvQ}.

Let $\Bdel$ be the set of $\obfx \in \Hd(q)$ such that
 \be\label{devdelta}\dev(\obfx)= \frac{|\rmr_q^{-1}(\obfx)|}{|\Hd|/|\Hd(q)|}- 1 > \delta.\ee
Suppose that $|\Bdel| \geq \eta |\Hd(q)|$. We will derive a contradiction for fixed $\delta, \eta$ and large enough $d$.  
 A similar bound applies to the set of $\obfx$ for which $\dev(\obfx) < - \delta$; taken together these yield Theorem \ref{mvQ}. 
 
If $\Bscr$ is a subset of $\Hd(q)$, and $\gamma_{\x}^{(\ell)}$ is the trajectory of an $\x$ chosen randomly -- with respect to counting measure -- from $\Hd$, the expected size of $\gamma_{\x}^{(\ell)} \cap \Bscr$ is just the sum over $i \in [-\ell \ldots \ell]$ of the probability that $\obfx_i$ lies in $\Bscr$; this probability is $|{\rmr_q^{-1}(\Bscr)|}/|\Hd|$, independently of $i$.  In other words, 

\begin{equation} \label{iow} \frac{1}{|\Hd|}\sum_{\x\in\Hd}\frac{|\gamma_{\x}^{(\ell)} \cap \Bscr|}{2\ell+1}= \frac{|\rmr_q^{-1}(\Bscr)|}{|\Hd|}.\end{equation}

We shall take $\Bscr = \Bdel$ and choose $\ell$ so that
\be\label{eq:rangel}
\frac{1}{5}|\Hd|<  q^25^{2\ell}\leq 5|\Hd|.
\ee

Note that we've used here the hypothesis that $q \leq d^{1/4 - \nu}$.

From \refs{devdelta} and \eqref{iow}, the average value of $\displaystyle\frac{|\gamma_{\x}^{(\ell)} \cap \Bdel|}{2\ell+1}$ as $\x$ ranges over $\Hd$ exceeds $$ \frac{|\Bdel|}{|\Hd(q)|}(1+\delta)\geq \frac{|\Bdel|}{|\Hd(q)|} + \delta \eta.$$ 

 Since $\frac{|\gamma_{\x}^{(\ell)} \cap \Bdel|}{2\ell+1}\leq 1$ for every $\x$, the number of $\x \in \Hd$ for which
\begin{equation}
\label{badprop} \frac{|\gamma_{\x}^{(\ell)} \cap \Bdel|}{2\ell+1} >    \frac{|\Bdel|}{ |\Hd(q)|} + \delta \eta /2 \end{equation}
is at least 

$$\frac{\delta \eta}{2} |\Hd| \gg_{\eps} \delta \eta d^{1/2-\eps}$$
for any $\eps>0$ (the last bound following from \refs{Hdsize}).  

Let $M$ be the number of non-backtracking marked paths on $\Hd(q)$ satisfying \eqref{badprop}.  Given that $\ell$ is chosen so that \refs{eq:rangel} is valid, we see from Proposition \ref{meansquare} that the number of pairs $(\x,\x')$ yielding the same trajectory $\gamma_{\x}^{(\ell)}$ on $\Hd(q)$ is not much larger than the number of diagonal pairs
$|\Hd|$: i.e. is bounded by $\ll_{\eps} d^{1/2+\eps}$. It follows that
\be
M \gg_{\eps,\delta,\eta} d^{1/2 - \epsilon}.
\ee
for any $\epsilon > 0$.

On the other hand, we also have an upper bound for $M$.  The total number of non-backtracking marked paths of length $2\ell$ on $\Hd(q)$ is on order of $5^{2 \ell - 1} |\Hd(q)| \ll_\eps d^{1/2 + \eps}$.  Proposition~\ref{chernoff} says that, of these, the proportion which satisfy \eqref{badprop} is at most $d^{-\tau}$ for some $\tau  = \tau(\delta, \eta) > 0$; so 
\be
M \ll_{\eps,\delta,\eta} d^{1/2 - \tau}
\ee
which yields a contradiction for $d$ sufficiently large.

\subsection{Hecke trees on the sphere: an idea of the proof of Theorems \ref{linnikinfty} and \ref{MVinfty}} We briefly sketch the proof
 of Theorems \ref{linnikinfty} and \ref{MVinfty}. The archimedean analogue of Proposition \ref{JLD} is the fact that the Hecke operator on $L^2(S^2)$
$$T_{5}.f:\ \x\mapsto \sum_{w\in\mcA_{5}}f(w.\x)$$ has a spectral gap: the eigenvalue $6$ has multiplicity one and any other eigenvalue
$\lambda$ satisfies $$|\lambda|\leq 2\sqrt 5;$$
this again is a consequence of Deligne's bounds for Fourier coefficients of holomorphic modular forms. 

Let $\mcT_{5}$ be the set of reduced words in the alphabet $\mcA_{5}$, which has the structure of a rooted $6$-valent tree. To any $\x$ in $S^2$, one associates  the infinite graph $\mcT_{5}.\x\subset S^2$, which is called {\em the Hecke tree}.  The most direct consequence of the spectral gap property
is that $\mcT_{5}.\x$ is dense in $S^2$ and even equidistributed: more precisely, uniformly on $\x \in S^2$,
the set $\{W_{\ell}.\x\}$ where $W_{\ell}$ ranges over the reduced words in $\mcT_{5}$
of length $\ell$ becomes equidistributed as $\ell\ra+\infty$ \cite{LPS1}. A more refined consequence is a large deviation estimate  for non-backtracking
marked paths of length $2\ell$ --uniform in the origin of the path--which is proven along the same lines as Proposition \ref{chernoff}.

To prove Theorem~\ref{MVinfty}, one proceeds from here as in the proof of Theorem~\ref{mvQ},  replacing conditions like ``two trajectories stay equal mod $q$ for a long time" with ``two trajectories stay near each other on the sphere for a long time.''

 \begin{figure}
\centering
\resizebox{!}{5.5cm}{
\includegraphics{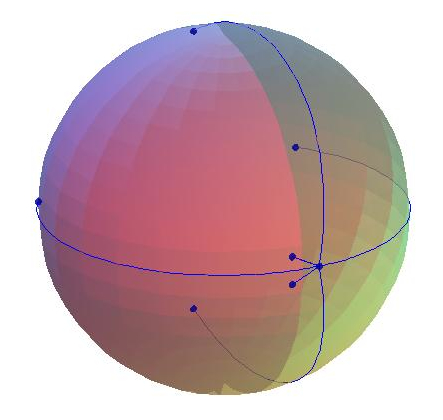}}

\centering
\resizebox{!}{5.5cm}{
\includegraphics{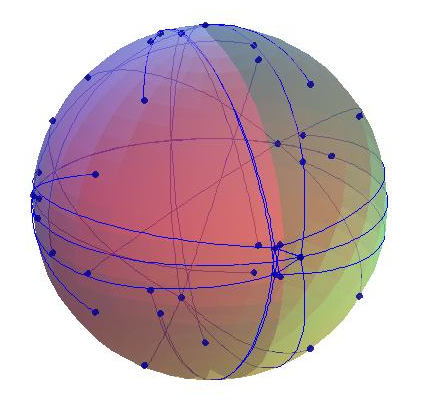}}

\centering
\resizebox{!}{5.5cm}{
\includegraphics{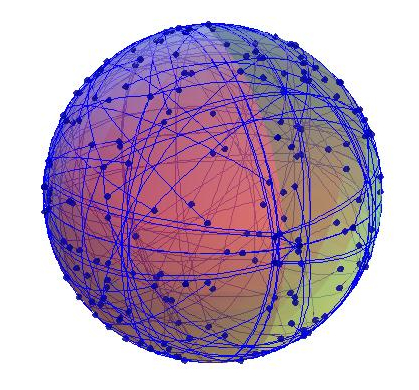}
}
\caption{The first three layers of the Hecke tree at $(10,1,0)$}
\end{figure} 
 \begin{figure}
\centering
\resizebox{!}{5.5cm}{
\includegraphics{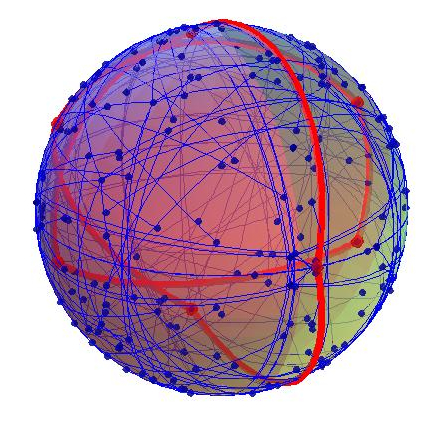}
}
\caption{The trajectory of $(10,1,0)$ within its Hecke tree}
\end{figure} 

\part{Classical theory}\label{Part:classical}
\section{The action of the class group on $\tildeHd$} \label{torsorA}

We present,  in Proposition \ref{torsor1}, a precise version of the homogeneous space structure on $\tildeHd$ discussed in Proposition \ref{torsor0}.   As we have remarked, the basic ideas here go back to Venkov \cite{Venkov,Venkov2} and in some sense to Gauss. 

It will be convenient to modify, slightly, the definition of $\tildeHd$ in the case when $d \equiv 3$ modulo $4$. 
Let $\SO_3(\Zz)^{+}$ be the index-$2$ subgroup of $\SO_3(\Zz)$ consisting of matrices which act on the coordinate lines via even permutations. Set 
$$\Hdstar=\begin{cases} \SO_3(\Z)^+\bash \Hd&\hbox{, if $d \equiv 1,2$ mod $4$}\\
\phantom{+}\SO_3(\Z)\bash \Hd&\hbox{, if $d \equiv 3$ mod $4$}
\end{cases}$$
Thus, $\Hdstar$ and $\tildeHd$ are equal when $d\equiv 3$ modulo $4$; otherwise,
the former is a double cover of the latter. 

 We shall show
that in fact $\Hdstar$ has the natural structure of torsor for $\Pic(\OO_K)$. 
Recall that a space $X$ homogeneous for the action of some group $G$ and for which the stabilizer of some (hence any) point is trivial is called a {\em principal homogeneous space} or {\em torsor} for $G$. In that case, for any $x\in X$, the map $$g\in G\mapsto g.x\in X$$ for any given $x\in X$ provides an identification of $G$ with $X$ as $G$-spaces. There is a priori no canonical way of choosing $x$; thus, one may think of a torsor as a set endowed with many different identifications with $G$ but, in general, with no {\em canonical} one.  For instance, the set of $n$th roots of $2$ is a torsor for the group $\mu_n$ of $n$th roots of unity. 

\subsection{Quaternions}\label{quaternions} We first recall the following classical facts.

Let $\B$ be the $\Qq$-algebra of Hamilton quaternions.  For $x=u+a.i+b.j+b.k\in\B$, the canonical involution is noted
 $\ov x=u-a.i-b.j-b.k$, the reduced trace $\tr(x)=x+\ov x=2u$ and the reduced norm $\Nr(x)=x.\ov x=u^2+a^2+b^2+c^2$.
 Let $\B^{(0)}$ denote the space of trace-free quaternions (the kernel of $\tr$) also called the {\em pure} quaternions. The space $\B^{(0)}$  endowed with the reduced norm is a quadratic space and the
 map \be\label{identification}(a,b,c)\mapsto a.i+b.j+b.k
 \ee
 is an isometry between the quadratic space $(\Qq^3,a^2+b^2+c^2)$ and $(\B^{(0)},\Nr)$.
  In the sequel, we will freely identify $\Q^3$ with $\B^{(0)}$, and, in particular, consider the elements of
   $\Hd$ as trace-free quaternions.

   We denote by $\B^\times$,  $\B^{1}$ and $\PB=\B^\times/Z(\B^\times)$ respectively , the group of units  of $\B$,  the subgroup of units of reduced norm one, and the projective group of units; these define $\Qq$-algebraic groups, and the action of $\B^\times$ on $\B^{(0)}$ by conjugation induces a covering and an isomorphism of $\Qq$-algebraic groups \cite[Th. 3.3]{Vigneras}:
\begin{equation} \label{Qcovering}Z(\B^\times)\hookrightarrow\B^\times\twoheadrightarrow\PB\stackrel{\sim}{\ra} \SO(a^{2}+b^2+c^2). \end{equation}

\subsection{Integral structures}\label{integralstructure}

Let $\B(\Zz)$ denote the ring of {\em Hurwitz quaternions},
$$\B(\Zz)=\Zz[i,j,k, \frac{1 + i + j + k}2];$$
It is well known that the ring of Hurwitz quaternions, endowed with the reduced norm $\Nr$, is {\em euclidean}:
for any $y,q\in\B(\Zz)-\{0\}$, there is $x, r\in\B(\Zz)$ such that $\Nr(r)<\Nr(q)$ and $y=qx+r$ (\cite[\S 5.7]{samuel}). 
This implies that any left (or right) $\B(\Zz)$-ideal is a principal ideal: any finitely generated left (resp. right) $\B(\Zz)$-module $I\subset \B(\Q)$ is of the form $\B(\Zz)q$ (resp. $q\B(\Zz)$) for some $q\in\B(\Q)$; moreover any subring of $\B(\Q)$ which is finitely generated (as a $\Zz$-module) is conjugate to a subring\footnote{Indeed, given $R$ such a subring, $R\B(\Zz)$ is a right $\B(\Zz)$ ideal, so of the form $R\B(\Zz)=q\B(\Zz)$ and $q^{-1}Rq\subset q^{-1}RR\B(\Zz)=q^{-1}R\B(\Zz)=\B(\Zz)$.} of $\B(\Zz)$; in particular $\B(\Zz)$ is a maximal order of $\B(\Q)$ and any maximal order in $\B(\Q)$ is $\B(\Q)^\times$-conjugate to it.

 Finally, under \refs{identification},  the lattice $\Zz^3\subset\Qq^3$  becomes identified with the trace free integral quaternions,
 $$\B^{(0)}(\Zz)=\B^{(0)}(\Qq)\cap\B(\Zz).$$
 
 In particular, we obtain a map $\B(\Zz)^{\times} \rightarrow \SO_3(\Zz)$ whose image is the index two subgroup $\SO_{3}(\Zz)^+$ and, similarly, $(\B(\Zz)\otimes_{\Zz}\Zz_{p})^\times \rightarrow \SO_3(\Zz_p)$ (which is surjective unless $p=2$).

While our primary interest will be in the Hamilton quaternions and sums of three squares, the reader will observe that much of what we state here and below is valid for more general quaternion algebras endowed with a maximal order. As this may be useful for readers interested in representations by other ternary quadratic forms, we will write several of the intermediate steps in more generality.

\subsection{Construction of representations using ideal classes}\label{generatepoints} As a first example of the usefulness of the quaternion, let us show how to deduce the Gauss-Legendre theorem 
from the {\em Hasse-Minkowski local-global principle}. If $d>0$ is not of the form $4^{a}(8b-1)$, then, for any prime $p$,
$d$ is representable as a sum of three squares in $\Zp^3$ and since $d$ is positive, $d$ is also representable over $\Rr$. By the
{\em Hasse-Minkowski theorem} (cf. \cite{serre}[Thm. 8, p. 41]), there exists $x=(a,b,c)\in\Qq^3$ such that $a^2+b^2+c^2=d$. Let $x=ai+bj+ck$; then $x^2=-d$ so the ring $\Zz[x]=\Zz+\Zz x$ is finitely generated. By \S \ref{integralstructure},  there is $q\in\B(\Q)$ such that $q\Zz[x]q^{-1}\in\B(\Zz)$
so that $y:=qxq^{-1}\in\B^{(0)}(\Zz)$ is integral and satisfies $y^2=-d$.

More generally, the above scheme together with the group of $\OO_{K}$-ideal classes makes it possible to generate plenty of {\em new} integral representations from a given one.

Any element $x\in \Hd$ yields an embedding of $\Qq(\sqrt{-d})$ into $\B(\Qq)$:
indeed $x^2=-d$ and thus $\sqrt{-d} \mapsto x$ defines an embedding  $$\iota_{x}:K\mapsto\Qq[x]\subset\B(\Qq).$$
 This embedding is {\em integral} in the following sense: let
$$\OO_x = \B(\Zz) \cap \Q[x]$$
then $\iota^{-1}_{x}(\OO_x)= \OO_{K}$  is the ring of integers\footnote{It is an order containing $\Z[\sqrt{-d}]$, integrally closed at
$2$, because the local order $\B(\Zz) \tensor_\Z \Z_2$ contains all elements of $\B \tensor_\Q \Q_2$ with norm in $\Z_2$. } of $K$.

Now, given such a $x$ and given an $\OO_{K}$-ideal, $I$, we can also construct a {\em new} integral representation $y\in\Hd$
from $x$ and $I$: the finitely generated $\Zz$-module $\B(\Zz)\iota_{x}(I)$ is a left $\B(\Zz)$-ideal, so of the form $\B(\Zz)q$ and then
$$y=qxq^{-1}\subset \B(\Zz)\iota_{x}(I)xq^{-1}=\B(\Zz)\iota_{x}(xI)q^{-1}\subset \B(\Zz)\iota_{x}(I)q^{-1}=\B(\Zz);$$
moreover if $I$ is replaced by $\lambda.I=I.\lambda$ $\lambda\in K^\times$, $q$ may be replaced by $q'=q\iota_{x}(\lambda)$
and $q'xq'^{-1}=q'\iota_{x}(\lambda)x\iota_{x}(\lambda^{-1}){q'}^{-1}=y$. Notice that $q$ is defined only up to multiplication on the left by
an element of $\B^\times(\Zz)$ this implies $y$ is well defined up to $\B^\times(\Zz)$-conjugacy: in view of the isomorphism $\B^\times(\Zz)/\pm 1\simeq \SO_{3}(\Zz)^+$, we obtain for $x\in\Hd$ fixed, a well defined map
$$[I]\in\Pic(\OO_{K})\mapsto \Hdstar.$$
Let us see that this map is a $\Pic(\OO_{K})$-torsor structure on  $\Hdstar$.

\subsection{}
For each pair of elements $x,y \in \HH_d$ we define the abelian group
\beq
\Lambda_{x \ra y} = \{ \lambda \in \B(\Zz): x  \lambda =   \lambda y  \}.
\eeq
Then $\Lambda_{x \ra y}$ has the structure of a module under $\OO_x$
(via $\mu \in \OO_x: \lambda \mapsto \mu  \lambda $) and also under $\OO_y$
(via $\nu \in \OO_y: \lambda \mapsto \lambda \nu $).  Both structures are torsion-free, and thus locally free.

  In fact,in view of the isomorphism $\SO_3(\Q) \cong \PB(\Qq)$,  there exists a $q \in \B^{\!\times}\!(\Q)$ such that $q^{-1}xq = y$.  (By Witt's theorem, any two elements of $\Q^3$ of the same length can be rotated into one another.)   One can then write $\Lambda_{x \ra y}$ as $\Q[x]q \cap \B(\Zz)$.  It follows that $$\Lambda_{x \ra y} \tensor_\Z \Q = \Q[x] q,$$ and thus that both module structures on $\Lambda_{x \ra y}$ are, indeed,  locally free of rank $1$.

We note that the map $M \rightarrow M    \otimes_{\OO_x}   \Lambda_{x \ra y}  $
gives a map from $\OO_x$-modules to $\OO_y$-modules; this induces an isomorphism
$\Pic(\OO_x) \ra \Pic(\OO_y)$, which is in fact the one induced
by the isomorphism $\OO_x \cong \OO_y$ sending $x$ to $y$. 
We denote by $[\Lambda_{x \ra y}]$ the class of $\Lambda_{x \ra y}$ in $\Pic(\OO_{x}) = \Pic(\OO_K)$.

\begin{prop}[Torsor structure on $\Hdstar$]\label{torsor1}
The set $\Hdstar$ has the structure of a torsor for $\Pic(\OO_{K})$ in which
the unique element of $\Pic(\OO_K)$ mapping $x$ to $y$ is given by 
$[\Lambda_{x \ra y}]$.

This action of $\Pic(\OO_K)$
descends to $\tildeHd$ (obvious if $d\equiv 3(\modu 4)$), and for $d\equiv 1,2(\modu 4)$ the stabilizer of any point of $\tildeHd$ is the order $2$ subgroup generated by a prime above $2$.

More precisely, 
for any $x,y, z \in \Hd$, 
\begin{enumerate}%
\item[A.] $\Lambda_{x \ra y} \otimes_{\OO_x} \Lambda_{y \ra  z}$ is isomorphic to
$\Lambda_{x \ra z}$ as an $(\OO_x, \OO_z)$-bimodule; 
\item[B.] The class of $\Lambda_{x \ra y}$ is trivial if and only if $x$ and $y$ are identified in
$\Hdstar$; 
\item[C.] For every $x$ and every $g \in \Pic(\OO_K$), there exists a $y$ such $[\Lambda_{x \ra y}] = g$.
\item[D.] If $d \equiv 1,2$ modulo $4$,  and $x \neq y \in \Hdstar$ project to the same element of $\tildeHd$, 
then $[\Lambda_{x \ra y}]$ is the ideal class of a prime above $2$. 

\end{enumerate}
\end{prop}
Note that parts A--D imply the first statements. 

In order to prove Proposition \ref{torsor1}, we shall need a study of the corresponding {\em local} problem,
which we undertake in \S \ref{torsor-local}.

\subsection{Some local analysis} \label{torsor-local}

Let $B$ be a quaternion algebra over $\Q_p$, let $\OO$ be a maximal order of $B$ and $\OO^{(0)}\subset B^{(0)}$ be the lattice of trace zero elements in $\OO$. As above, conjugation by element of $B^\times$ induces a surjective map $$B^\times\twoheadrightarrow \SO(B^{(0)},\Nr)$$ under which $\OO^\times$ maps to $\SO(\OO^{(0)},\Nr)$.
Let $\OO^{(1)} \subset \OO^\times$ be the group of norm-$1$ units.  Let $d \in \Zp$ be squarefree (that is, $\ord_p(d) \leq 1$)  and let $\OO^{(0,d)}\subset\OO^{(0)}$ be the set of  elements of $\OO$ with
trace $0$ and norm $d$;  $\OO^{(0,d)}$ is obviously stable under the action of $\OO^{\times}$ by conjugation.

We fix an element $x$ of $\OO^{(0,d)}$ and, as above, write $\Lambda_{x \ra y}$ (or, when no confusion is likely, just $\Lambda$) for the set of $\lambda \in \OO$ such that
\begin{equation}
x \lambda = \lambda y.
\label{eq:lambdadef}
\end{equation}
The solutions to \eqref{eq:lambdadef} in $B$ form a vector space of dimension $2$, so $\Lambda$ is a free $\Z_p$-module of rank $2$. %

\begin{prop}
\label{pr:localaction}
The action of $\OO^{(1)}$ and $\OO^{\times}$ on $\OO^{(0,d)}$ can be described as follows.

\begin{itemize}
\item[-] Suppose $B$ is a division algebra.  
Then there are two orbits of $\OO^{(1)}$ on $\OO^{(0,d)}$, which are interchanged by conjugation by any element of $\OO^{\times}$ whose norm is not in $\Nr(\Qp[x])$.  In particular, the special orthogonal group of $\OO^{(0)}$ acts transitively on $\OO^{(0,d)}$.

\item[-] Suppose $B = M_2(\Q_p)$. Then:
\begin{itemize}
\item If $p \neq 2$ and $p$ does not divide $d$,  the action of $\OO^{(1)}$ on $\OO^{(0,d)}$ is transitive. 
\item Otherwise, there are two orbits of  $\OO^{(1)}$ on $\OO^{(0,d)}$;
they are interchanged by $\OO^{\times}$, unless $p=2$ and $d \equiv 3 $ mod $4$. 
\end{itemize}

\end{itemize}

\end{prop}

\begin{proof}

Suppose $B$ is a division algebra.  Then $\OO$ is the unique maximal order, and consists of all elements whose norm lies in $\Zp$.  Let $\lambda$ be a nonzero element of $\Lambda$; then $y = \lambda^{-1} x \lambda$.  Note that conjugation by $\lambda$ is the desired isometry of $\OO^{(0,d)}$ relating $x$ to $y$. 

 If there is an element $\alpha \in \Qp[x]$ with $\Nr(\alpha) = \Nr(\lambda)$, then $\alpha^{-1}\lambda $ is an element of $\Lambda$ of norm $1$; conversely, any element of $\Lambda$ of norm $1$ is $\alpha^{-1}\lambda$ for some $\alpha \in \Qp[x]$ whose norm agrees with that of $\lambda$.  This shows that the orbits of $\OO^{(1)}$ on $\OO^{(0,d)}$ are naturally identified with $\Qp^{\times} / \Nr(\Qp[x]^\times)$.  This quotient is a group of order $2$. 

Now suppose that $B = M_2(\Q_p)$, so that we can take $\OO^{\times} = \GL_2(\Z_p)$ and $\OO^{(1)} = \SL_2(\Z_p)$.

Let $x = \mat{b}{a}{c}{-b}$ be an element of $\OO^{(0,d)}$ (so that $b^2 + ac = -d$) and $y = \mat{0}{1}{-d}{0}$; then 
\beq
\Lambda_{x \ra y} = \Z_p \mat{b}{1}{c}{0} + \Z_p\mat{a}{0}{-b}{1}
\eeq
and the elements of $\Nr(\Lambda_{x \ra y})$ are those elements of $\Z_p$ represented by the quadratic form $Q  = aX^2 + 2bXY - cY^2$, which has discriminant $-4d$.  Thus, $x$ and $y$ are in the same orbit of $\OO^{(1)}$ if and only if $Q$ represents $1$ over $\Zp$.

For all facts used below about isomorphism classes of binary quadratic forms over $\Z_p$, see  ~\cite[\S 31]{jone:qf}. 

First, suppose $p$ is odd.   If $p$ does not divide $d$, then $Q$ is equivalent to $X^2 + dY^2$, and in particular represents $1$.  So in this case, $\OO^{(1)}$ acts transitively on $\OO^{(0,d)}$.  If $p$ divides $d$, then $Q$ is equivalent to either $X^2 + dY^2$ or $\eps X^2 + \eps^{-1} d Y^2$, where $\eps \in \Zp^{\times}$ is a nonsquare.  In the former case, $y$ is in the orbit of $x$; in the latter case,
\beq
y' = \mat{0}{\eps}{-\eps^{-1} d}{0}
\eeq
is in the orbit of $x$.  So there are two orbits, as claimed.  In both cases, $Q$ represents an element of $\Zp^{\times}$, so $\OO^{\times}$ acts transitively on $\OO^{(0,d)}$.

Now take $p=2$.  In this case, there are always exactly two equivalence classes of binary forms of discriminant $4d$, one of which represents $1$ over $\Z_2$ and the other of which does not.  The non-representing forms are:
\begin{itemize}
\item $2X^2 + 2XY + (1/2)(d + 1) Y^2		\hfill  	(d = 3 \mod 4)$
\item $\eps X^2 + \eps^{-1} d Y^2 \hfill (d =1,2 \mod 4)$
\end{itemize}
where, in the latter case, $\eps \in \Z_2^{\times}$ is an element which is not a norm from $\Q_2(\sqrt{-d})^{\times}$.  In case $d=1$ or $2$ mod $4$, we again see that either $y$ or  $y' = \mat{0}{\eps}{-\eps^{-1} d}{0}$ lies in the orbit of $x$, so there are two orbits of $\OO^{(1)}$ on $\OO^{(0,d)}$.  And again $Q$ represents an element of $\Z_2^{\times}$, so some element of $\OO^{\times}$ sends $x$ to $y$.

In case $d = 3$ mod $4$, take
\beq
y' = \mat{1}{-2}{(1/2)(d+1)}{-1}
\eeq
A direct computation shows that the orbit of $x$ under $\OO^{(1)}$ contains either $y$ or $y'$, so again there are two orbits.  In this case, the two orbits are not interchanged by $\OO^{\times}$.  
\end{proof}

\subsection{Proof of Proposition \ref{torsor1}, A  } \label{Gcdproof}

We define a map
\beq
m: \Lambda_{x \ra y} \tensor_{\OO_y}  \Lambda_{y \ra z}  \ra \Lambda_{x \ra z}
\eeq
by sending $\lambda \tensor \mu$ to $\lambda \mu$.  Evidently it preserves
the natural structures of $(\OO_x, \OO_z)$-bimodule on the two sides. We will
show that the image of $m$ -- call it $\Lambda'$ -- coincides
with $\eta \Lambda_{x \ra z}$, for some $\eta \in \OO_x$. This will imply our conclusion, 
because ``multiplication by $\eta^{-1}
$'' is also an isomorphism of $(\OO_x, \OO_z)$-bimodules.

Clearly $\Lambda' \subset \Lambda_{x \ra z}$. 
The $\Z_p$-module $\Lambda_{x \ra y} \tensor_\Z \Z_p$ is explicitly described by Proposition~\ref{pr:localaction}.  
If $p$ is odd, it follows from  from the second case of Proposition~\ref{pr:localaction} that $\Lambda_{x \ra y}, \Lambda_{x \ra z}, \Lambda_{y \ra z}$ all include elements of $\B(\Zp)=\B(\Zz) \tensor_\Z \Z_p$ whose norm lies in $\Zp^{\times}$.  
It follows that $\Lambda'$ also contains such an element. 
If $p=2$, the same conclusion holds unless $d \equiv 3$ mod $4$;
if $d \equiv 3$ mod $4$,  we may conclude only that $\Lambda_{x \ra y}$ contains an element
whose norm has valuation $\leq 1$.

Both $\Lambda' $ and $ \Lambda_{x \ra z}$ are free rank $1$ modules under $\OO_x$; thus, for every $p$, there exists $\eta_p \in \OO_x \otimes \Z_p$ such that
$\Lambda' \otimes \Z_p = \eta_p (\Lambda_{x \ra z} \otimes \Z_p)$. 
In order that both localizations contain an element whose norm belongs to $\Z_p^{\times}$, 
it is necessary that $\eta_p$ belong to $(\OO_x \otimes \Z_p)^{\times}$;
this implies that $\Lambda' \otimes \Z_p = \Lambda_{x \ra z} \otimes \Z_p$. 

  Thus, unless $d \equiv 3$ modulo $8$, 
we indeed have an equality $\Lambda' = \Lambda_{x \ra z}$. 
In the case where $d \equiv 3$ modulo $8$, we may still draw the
conclusion that there exists an ideal $I$ in $\OO_K$, 
divisible only by the prime above $2$, so that $\Lambda' = I . \Lambda_{x \ra z}$. 
This implies the stated conclusion, since the prime of $\Q(\sqrt{-d})$ above $2$ is principal when $d \equiv 3$ modulo $8$.

\subsection{Proof of Proposition \ref{torsor1}, B   }

Let us recall first of all the following: for $x,y \in \mathscr{H}_d$,
\begin{eqnarray*} x \in \SO_3(\Z)^+ y \iff  u x u^{-1} = y, \mbox { some $u \in \B(\Zz), \Nr(u)=1$}. \\
 x \in \SO_3(\Z) y \iff  u x u^{-1} = y, \mbox { some $u \in \B(\Zz), \Nr(u) \in \{1,2\}$}. \end{eqnarray*}
These facts may be verified by direct computation. 

Notice that if $\Lambda_{x \ra y}$ is trivial
there exists a global generator $\lambda$ for $\Lambda_{x \ra y}$ as an $\OO_x$-module.
By the conclusions of the prior paragraph \S \ref{Gcdproof}, the norm of $\lambda$ is a $p$-adic unit
for all $p > 2$.

We separate two cases, as in the prior proof. 

{\em Case 1:  $d = 1,2 (4)$}. Suppose that  $\Lambda_{x \ra y}$ is trivial.
In this case, the norm of $\lambda$ is also a unit at $p=2$. 
Thus, there exists a unit $u \in \B(\Zz)^\times$ such that $uxu^{-1} = y$. 

Conversely, if there exists a unit $u$ with $uxu^{-1} = y$, then $u$ lies in $\Lambda_{x \ra y}$.  The span $u. \OO_x$
is a $\OO_x$-submodule of $\Lambda_{x \ra y}$, and coincides with it by the same reasoning
as used in \S \ref{Gcdproof}. Therefore, $\Lambda_{x \ra y}$ is principal as an $\OO_x$-module.

{\em Case 2: $d = 3 (8)$}.  if $\Lambda_{x \ra y}$ is trivial.
In this case, the $2$-valuation of the norm of $\lambda$ is at most $1$. Therefore, 
$y = \lambda x \lambda^{-1}$ with $\lambda$ an element of $\B(\Zz)$ of norm $1$ or $2$.

Conversely, suppose that $x,y \in \mathscr{H}_d$ are related by an element of $\SO_3(\Z)$.  Equivalently, there exists $\lambda \in \Lambda_{x \ra y}$ with $\Nr(\lambda) \in \{1,2\}$.  If $\Nr(\lambda) = 1$, then $[\Lambda_{x \ra y}]$ is principal, exactly as above.  Suppose $\Nr(\lambda) = 2$.  
Reasoning as before, $\lambda. \OO_x$ is an $\OO_x$-submodule of $\Lambda_{x \ra y}$
of index $\leq 2$; since the ideal above $2$ in $\OO_x$ is principal, 
we deduce again that $[\Lambda_{x \ra y}]$ is trivial.

\subsection{Proof of Proposition \ref{torsor1}, C  }
\label{ss:picokacts}
Write $\widehat{\OO}_K$ for $\OO_K \tensor_\Z \whZ$.  Given  $\alpha \in \widehat{\OO}_K$ everywhere locally invertible, one defines the $\OO_{K}$-ideal $(\alpha) = \alpha \widehat{\OO}_K \cap \OO_K$.  Choose $\alpha$ such that $(\alpha)$ is in the class of $g \in \Pic(\OO_K)$.

Choosing an $x \in \HH_d$, and thus an embedding $\iota_{x}$ of $K \simeq \Q[x]$ in $\B$; using $\iota_{x}$ to identify $\alpha$ with an element of $\Qq[x]\otimes_{\Zz}\whZ$ we consider the left $\B(\Zz)$-ideal  $\B(\Zz)(\alpha)$ and argue as in \S \ref{generatepoints}; this ideal is principal (as $\B(\Zz)$ is euclidean, \S \ref{integralstructure}) so there exists $q \in \B^\times(\Q)$ so that $\B(\Zz) q^{-1} = \B(\Zz) (\alpha)$.  In particular, we have $q =  \alpha^{-1} u$ with $u \in \widehat{\B(\Zz)}^\times$.

Now let $y = q^{-1} x q =u^{-1} \alpha x\alpha^{-1}  u= u^{-1} x u$;  then $y$ lies in $\widehat{\B(\Zz)}$ and in $\B(\Qq)$, so $y \in \Hd$.

 Now \begin{equation} \label{above} \Lambda_{x \ra y} =  \Q[x] q\cap \B(\Zz) \cong \Q[x] \cap \B(\Zz) q^{-1}
 = \Q[x] \cap \B(\Zz) (\alpha) = (\alpha).\end{equation} 
So $[\Lambda_{x \ra y}] = g$, which proves assertion (3).

\subsection{Proof of Proposition \ref{torsor1}, D}  
Notation as in the statement of Proposition \ref{torsor1}D. 
The proof of \ref{torsor1}, B shows that $[\Lambda_{x \ra y}]$ is a prime ideal above $2$. 
It is not principal since $x \neq y$. Since $d \equiv 1,2$ modulo $4$, 
the prime $2$ is ramified in $\Q(\sqrt{-d})$, and therefore this uniquely specifies
$[\Lambda_{x \ra y}]$. 
 
\subsection{The $\Pic(\OO_{K})$-action explicated}\label{ss:explicitaction}
The arguments of \S \ref{ss:picokacts} enables us to describe the action of $\Pic(\OO_K)$ on $\Hdstar$
or $\tildeHd$ quite explicitly.  Let $\ic{p}$ be an ideal of norm $5$ in $\OO_K$; we now verify that the algorithm prescribed in 
\S \ref{explication5} is indeed a lifting of the action of $[\ic{p}]$ from $\Hdstar$ to $\Hd$.

Let $x \in \Hd$, thus giving an embedding of $\iota_{x}$ of $K$ into $\B$. 
Let $\alpha$ be an element of $\widehat{\OO}_K$ which projects to a unit in $\OO_{K_v}$ when $v$ is prime to $\ic{p}$, and to  
a uniformizer in $\OO_{K_\ic{p}}$.  Choose $q \in \B^\times(\Q)$ such that $\B(\Zz) q^{-1} =  \B(\Zz)(\alpha)$; then $q$ has  norm $5^{-1}$.

Set\footnote{The motivation for introducing this may become clearer later; 
the lifting of the action from $\Hdstar$ to $\Hd$ involves imposing additional congruences modulo $3$, and $q_2$ will have better mod $3$ properties. }  $q_2 = q(1+i)^{-1}$, which has norm $1/10$. Taking $y = q_2^{-1}x q_2$, we have, as in \eqref{above}, 
\beq
\Lambda_{x \ra y} \cong \Q[x] \cap  \B(\Zz) (1+i) (\alpha)
\eeq
The right-hand side is an $\OO_x$-submodule of $\ic{p}$ of index $2$. 
If $d$ is $1$ or $2$ mod $4$, this means that $\Lambda_{x \ra y}$ is in the class of $\ic{q}\ic{p}$, where $\ic{q}$ is the ideal of $\OO_x$ lying over $2$.  If $d$ is $3$ mod $4$, we have
\beq
\Q[x] \cap \B(\Zz) (1+i) (\alpha)  = \Q[x] \cap  \B(\Zz) (2 \alpha) \cong \Q[x] \cap  \B(\Zz) (\alpha)
\eeq
so the class of $\Lambda_{x \ra y}$ is just that of $\ic{p}$.
We note, however, that the action of $\ic{p}$ and $\ic{p} \ic{q}$ on $\tildeHd$ is the same,
since the action of $[\ic{q}]$ permutes the fibers of $\Hdstar \rightarrow \tildeHd$.

The integral quaternions of norm $10$ can all be expressed in the form $ru$, where $u \in \B(\Zz)^\times$ and $r$ is an element of 
\be\label{T5quaternions}
\mcA_{5}=\{1 \pm 3i, 1 \pm 3j, 1 \pm 3k\}.
\ee

It follows that, for each $x$ in $\Hd$, the class $[\ic{p}]\tilde x$ in $\tildeHd$ must be represented by $r^{-1} x r$ for some $r\in \mcA_5$.  Now a direct computation shows  that the action of conjugation by the six elements of $\mcA_5$ yields precisely the action of the six matrices appearing in \S \ref{primeaction}.  For example,   
conjugation by $1-3i$ acts on $\B^{(0)}$ via the rule
\begin{align*}
i &\mapsto i,\\
 j &\mapsto \frac{1}{10} (1+3i) j (1-3i) = -\frac{4}{5}j+\frac{3}{5}k,\\
  k&  \mapsto \frac{1}{10} (1+3i) k(1-3i) = -\frac{3}{5}j -\frac{4}{5}k
  \end{align*}
which corresponds to the matrix $A$ in \S \ref{explication5}.

\section{Representations of binary quadratic forms by $x^2 + y^2 + z^2$.}
\label{linniklemma}

We now discuss the proof of Proposition \ref{BasicLemma}, ``Linnik's basic lemma.''
We shall, in fact, discuss two proofs; the first (\S \ref{pallsec}) is simply quoting a result of G. Pall,
which in turn rests on Siegel's mass formula; the second, presented in \S \ref{Pall2proof} 
after a preliminary discussion, is based on ideas related to \S \ref{torsorA}. 
\subsection{A result of Gordon Pall.} \label{pallsec} Let $(Q,\Zz^m)$, $(R,\Zz^n)$ be two non-degenerate integral quadratic forms with $m\geq n$. One says that $Q$ {\em represents} $R$
 if there exists a $\Z$-linear map
$\iota:\Zz^n\ra\Zz^m$ such  that, for any $\bfx\in\Zz^n$, $Q(\iota(\bfx))=R(\bfx)$.  It follows that $\iota$ is an embedding.

In particular, given $\bfx_{1},\bfx_{2}\in\Hd$ with $\bfx_{1}.\bfx_{2}=e$, 
the linear map $$\iota:\Zz^2\ra\Zz\bfx_{1}+\Zz\bfx_{2}$$ defines a representation of the binary quadratic form
 $R(x,y)=dx^2+2e xy+dy^2$ by the ternary form
$Q(x,y,z)=x^2+y^2+z^2$.

Therefore, counting  the number 
 of pairs $(\bfx_{1},\bfx_{2})\in\Hd^2$ with $\bfx_{1}.\bfx_{2}=e$ is essentially equivalent (up to the action of the finite group $\SO_{3}(\Zz)$)
  to counting the number of representations of $R$ by $Q$.
  
The precise computation of the number of embeddings, $r(a,b,c)$ say, of a binary quadratic
form $a x^2 + b x y + c y^2$ into the ternary quadratic $x^2+y^2+z^2$ 
in the {\em most general case} (no primitivity assumptions, etc.) was first carried out by Pall~\cite[Theorem 4, page 359]{Pall}.  His theorem gives a formula

\begin{equation} \label{pallthm} r(a,b,c) = 24 \cdot 2^{\nu} \cdot \prod_{p|2(b^2-4ac)} r_{p}(a,b,c)\end{equation}
where it follows from Pall's result that:
\begin{enumerate}
\item $\nu$ is the number of distinct odd primes dividing the discriminant $b^2-4ac$;
\item $r_{p}(a,b,c)$ is bounded by an absolute constant unless $p^2|(a,b,c)$. 
\end{enumerate}
In particular, it follows that \begin{equation}  \label{pall} |r(a,b,c)| \ll \max(a,b,c)^{\varepsilon},\end{equation} when $(a,b,c)$ has no square factor.

One way to obtain a quantitative result like \eqref{pall} is by use of Siegel's mass formula, combined with a computation of local densities.   We will {\em sketch} in the rest of this section an alternate approach. 
We must first revisit the material of \S \ref{torsorA} and describe a ``different'' (although closely related, as we shall see) connection between $\Hd$ and the class group.

\subsection{The orthogonal complement construction} \label{orthocomplement}
We have seen that $\Hdstar$ can be placed in bijection with the group $\Pic(\OO_K)$; but there is no natural group structure on $\Hdstar$, for there is no natural choice of an identity element of $\Hdstar$. In other words, the torsor structure is natural but admits no natural trivialization. 
However, the orthogonal complement construction, which we shall now explain, gives a trivialization
of the ``square'' 
 $\Hdstar \times_{\Pic(\OO_K)} \Hdstar$. \footnote{The situation is similar to one familiar from geometry:
 If $X$ is a smooth cubic plane curve over a field $k$ with Jacobian $E$, then the points $X(k)$ form a torsor for the group $E(k)$.  This torsor has no canonical trivialization, but the embedding of $X$ into the plane gives a canonical trivialization of the "cube" $X \times_E X \times_E X$. }

Let $\mathcal{Q}_{-d}$ be the set of binary quadratic
forms $aX^2+bXY+cY^2$ of discriminant 
\beq
b^2-4ac = \disc(\order_{K}) = \begin{cases} -d, \ d \equiv 3 \mbox{ mod $4$} \\ -4d, \ d \equiv 1,2 \mbox{ mod $4$} \end{cases}
\eeq
considered up to $\SL_2(\Z)$-equivalence.  Then $\mathcal{Q}_{-d}$ parametrizes rank-$2$ quadratic lattices $\Lambda$ of discriminant $\disc(\order_K) \in \{-d,-4d\}$, endowed with an orientation $\wedge^2 \Lambda \cong \Z$.

There is a natural map
$\Hdstar \ra \mathcal{Q}_{-d}$
sending $x \in \Hd$ to the induced quadratic form on $x^\perp$ (the rank $2$ lattice in $\Zz^3$ of vectors orthogonal to $x$); 
here we understand that, if $d \equiv 3$ modulo $8$, then we scale the resulting quadratic form by $\frac{1}{2}$, and that a basis $(\lambda,\lambda')$ of $x^\perp$ is  oriented if $\lambda\wedge\lambda'\wedge x>0$.

Composing with the classical identification of $\mathcal{Q}_{-d}$ with $\Pic(\OO_K)$ now gives a map
\begin{equation}
\Perp:  \Hdstar \ra \Pic(\OO_K)
\end{equation}
which can be described concretely as $\Perp(x) = [\Lambda_{x \ra \bar{x}}]$.  Both sides admit an action of $\Pic(\OO_K)$; the left-hand side by means of the torsor structure described in section \S~\ref{quaternions}, and the right-hand side by multiplication.  But $\Perp$ is not equivariant for this action; rather, it intertwines the action of $\Pic(\OO_K)$ on the left with the {\em square} of this action on the right, as we will now see.  Suppose given $x, y$ in $\Hd$ such that $[\Lambda_{x \ra y}] = \alpha \in \Pic(\OO_K)$.  One checks that $[\Lambda_{y \ra x}] = [\Lambda_{\bar{x} \ra \bar{y}}]$.  Thus
\beq
\Perp({y}) = 
[\Lambda_{y \ra \bar{y}}] = 
[\Lambda_{y \ra x}] [\Lambda_{x \ra \bar{x}}]  [\Lambda_{\bar{x} \ra \bar{y}}]
= \alpha^2 \Perp(\bar{x}).
\eeq

It follows that the image of $\Perp$ is precisely one coset of $2\Pic(\OO_K)$.
\footnote{It is irresistible to ask {\em which} coset. 
A quadratic form  $Ax^2+Bxy+Cy^2$  of discriminant $d$ embeds, over $\Q$, 
into $\Z^3$ with the standard quadratic form only if $(A, d/A)_p = (d, d)_p$
for every $p$ dividing $2d$; 
here $(a,b)_p$ is the {\em Hilbert symbol}.  This condition in fact
defines a coset of squares, and so describes exactly the image of $\Perp$.}
Moreover, the cardinality of a fiber of $\Perp$ is just the order of the $2$-torsion subgroup of $\Pic(\OO_K)$.

What $\Perp$ supplies is not a trivialization of the torsor $\Hdstar$, but a trivialization of its {\em square} in the group of $\Pic(\OO_K)$-torsors.  The square is a torsor $T$ which can be described explicitly as the set of equivalence classes of pairs $(x,y)$ under the relation $(x,y) = (\alpha x, \alpha^{-1} y)$ for all $\alpha \in \Pic(\OO_K)$.  The action of $\alpha \in \Pic(\OO_K)$ sends $(x,y)$ to $(\alpha x, y)$ (or to $(x,\alpha y)$, which is the same.)  Then the map sending $(x,y)$ to $[\Lambda_{x \ra \bar{y}}]$ is a canonical isomorphism between $T$ and $\Pic(\OO_K)$.  

\subsection{Bounds for representations, revisited.}  \label{Pall2proof}
We now sketch how the ``orthogonal complement construction'' leads us to another proof of Proposition \ref{BasicLemma}. 

For simplicity, we consider only the case where $ b^2-4ac$ 
is of the form $-4d$, with $d$ squarefree.  
Given an embedding
$$\iota: (\Z^2, ax^2+bxy+cy^2) \hookrightarrow (\Z^3,x^2 + y^2 + z^2).$$
 consider the orthocomplement inside $\Z^3$ of $\iota(\Z^2)$.
It is generated by a single vector $\x \in \Z^3$, unique up to sign.  The quadratic lattice $\x^\perp$ is plainly the same as its sublattice $\iota(\Z^2)$ after tensoring with $\Q$; since $d$ is squarefree, the two lattices are in fact identical.  This implies that  $\x. \x =  d$. 

The embedding $\iota$ is determined up to at most $6$ possibilities ($6$ being the maximal number of automorphisms of a positive definite form in rank two)  by $\x$. It suffices,
therefore, to count the number of $\x \in \Hd$ so that the quadratic form induced on $\x^{\perp}$
is isomorphic to $(\Z^2, ax^2+bxy+cy^2)$. 
By \S \ref{orthocomplement},  this is at most $24  |\Pic(\OO_K)[2]|$, where $[2]$ denotes $2$-torsion.  By genus theory we have $|\Pic(\OO_K)[2]| \ll d^{\eps}$, and the bound \eqref{pall} follows.

\part{Adelization} \label{adelization} 
In this part, we interpret, in adelic terms, the various classical arithmetic sets and structures discussed so far --i.e. $\Hd,\ \tildeHd,\ \Pic(\order_{K}),\ \Hdq,\ \tildeHdq$ -- and the various maps between them. This interpretation gives us an alternative approach to the results described in Part \ref{Part:classical} and more importantly, will be key to the proof of Proposition \ref{JLD}. We refer to \cite{Knapp} for a recent gentle introduction to the ad\`elic theory of algebraic groups (in relation with automorphic forms), and to \cite{PSPM9,PR} for more extensive treatments.

We denote by $\A_{f}$ the ring of finite ad\`eles of $\Q$, i.e. the restricted product of $(\Qq_{p})_{p\mathrm{\ prime}}$ with respect to the sequence of maximal compact subgroups $(\Zz_{p})_{p\mathrm{\ prime}}$, and by $\A=\Rr\times\A_{f}$ the ring of 
ad\`eles. We denote by $\whZ=\prod_{p}\Z_{p}\subset \A_{f}$ the maximal compact subgroup of $\A_{f}$; alternatively $\whZ$ is the closure of $\Zz$ in $\A_{f}$. Given $\mathrm{G}$ a $\Qq$-algebraic group with a chosen model over $\Z$, we denote by $\mathrm{G}(\A_{f})$, $\mathrm{G}(\A)$, $\mathrm{G}(\whZ)$ the groups of points of $\mathrm{G}$ in the corresponding rings.

We begin by summarizing the contents of this part (\S \ref{sec:adelicHdq} -- \S \ref{sec:adelicredq})
via  a  commutative diagram:

We set (see \S \ref{sec:genus} below for the definition of $\gen_{\SO_{3}}(\Z^3)$)
\begin{eqnarray*} \mcP & =& \{ (L,\x),\     L \in \gen_{\SO_{3}}(\Z^3),\  \x \in L,\ \ \x.\x = d\}.\\
 \mcP_{(q)} & := & \{(L,\obfx),\     L\in\gen_{\SO_{3}}(\Z^3),\ \obfx\in L/qL,\ \obfx.\obfx\equiv d(\modu q)\}.  \end{eqnarray*}
For the sequel, we need to fix {\em base points} $\bfx_{0}\in\Hd$ and $\obfx_{q}\in\Hdq$; this being done, we define
\begin{eqnarray}\label{Kfdef} K_{f}[q ]&:=& \mathrm{ker}( \SO_3(\whZ) \rightarrow \SO_3(\Z/q\Z) )\\ 
\label{K'fdef}K'_f[q] &:=& \{ g \in \SO_3(\whZ):  g  . \obfx_{q} = \obfx_{q} \}\end{eqnarray}
and denote by $\SO_{\bfx_{0}}$ the stabilizer of $\bfx_{0}$ in $\SO_{3}$. We will show in the following sections how to describe $\tildeHd$ and $\tildeHdq$ adelically, by means of a commutative diagram
\be\label{massive-identificationtilde}
\xymatrix{
\tildeHd\ar[r]^-{\sim}\ar[d]_{\rmr_{q}}&\SO_3(\Q)\bash\P\ar[d]_{\rmr_{q}}&\SO_{\x_0}(\Q) \bash \SO_{\x_0}(\Aaf) / \SO_{\x_0}(\Zhat)\ar[l]_-{\sim}\ar[d]\\
\tildeHdq\ar[r]^-{\sim}&\SO_3(\Q)\bash\mcP_{(q)}&\SO_3(\Q) \bash \SO_3(\Aaf) / K'_f[q]\ar[l]_-{\sim}.
}\ee

We will also explain how to describe $\Hd$ and $\Hdq$ adelically, which is slightly more involved technically (though no different conceptually.)
 
For this we define
\begin{eqnarray*} \Pq &= &\{ (L,\x, \theta ), \ L \in \gen_{\SO_{3}}(\Z^3),\ \x \in L,\ \x.\x = d, \theta: L/3L \simeq (\Z/3\Z)^3\}. \\ 
 \mcP_{(3,q)} & := & \{(L,\obfx,\theta),\     L\in\gen_{\SO_{3}}(\Z^3),\ \obfx\in L/qL,\ \obfx.\obfx\equiv d(\modu q),\\
&& \hskip 7cm\ \theta:L/3L\simeq (\Zz/3\Zz)^3\},
 \end{eqnarray*}
 and, for any integer $a$,  
 \be\label{Kfaqdef}K_{f}[a,q]:=K_{f}[a]\cap K'_{f}[q]\subset K'_{f}[q].
 \ee
Then we will construct a diagram
 \be\label{massive-identification}
\xymatrix{
\Hd\ar[r]^-{\sim}\ar[d]_{\rmr_{q}}&\SO_3(\Q)\bash\Pq\ar[d]_{\rmr_{q}}&\SO_{\x_0}(\Q) \bash \SO_{\x_0}(\Aaf)\times\SO_{3}(\whZ/3\whZ) / \SO_{\x_0}(\Zhat)\ar[l]_-{\sim}\ar[d]\\
\Hdq\ar[r]^-{\sim}\ar@{>>}[d]&\SO_3(\Q)\bash\mcP_{(3,q)}\ar@{>>}[d]&\SO_3(\Q) \bash \SO_3(\Aaf) / K_f[3,q]\ar[l]_-{\sim}\ar@{>>}[d]\\
\tildeHdq\ar[r]^-{\sim}&\SO_3(\Q)\bash\mcP_{(q)}&\SO_3(\Q) \bash \SO_3(\Aaf) / K'_f[q]\ar[l]_-{\sim};
}\ee
where the vertical arrows at the bottom are the evident surjective maps. The
horizontal arrows of these diagrams are described in \S \ref{sec:adelicHdq} and \S \ref{sec:adelicHd},  while the vertical
arrows are discussed in \S \ref{sec:adelicredq}.
 
 \subsection{Convention: orthogonal groups vs. unit group of quaternions}
 We have already noted in \S \ref{quaternions} that the quadratic spaces $(\Qq^3,a^2+b^2+c^2)$, $(\B^{(0)},\Nr)$ are isometric and that the action --by conjugation-- of the units of Hamilton quaternions, $\B^\times$, on $\B^{(0)}$
 induces and isomorphism of $\Qq$-algebraic groups
 $$\PB\simeq \SO_{3}$$
 with $\PB=Z(\B^\times)\bash\B^\times$ the projective group of units. In particular the group of adelic points $\PB(\Aa)$
 and $\SO_{3}(\Aa)$ get naturally identified as are their various respective subgroups. In the sequel, we shall freely use this identification, moreover we will use the same notations for various subgroups $K_{f}[q],\ K'_{f}[q]$ etc. to denote either some subgroup in $\SO_{3}(\Aa)$ or its image in $\PB$ under  the above isomorphism.

\section{Adelic interpretation of $\Hd(q)$.}\label{sec:adelicHdq}

In this section  we identify $\Hd(q)$ (as well as the sphere $S^2$), with an adelic quotient of $\SO_{3}$, 
verifying the second lines of \eqref{massive-identificationtilde} and \eqref{massive-identification}.

\subsection{Adelic actions on rational lattices} 
The adelic group $\GL_{3}(\adele_{f})$ acts on the space of lattices in $\Q^3$ as follows.
If $L \subset \Q^3$ is a lattice, and $g=\prod_{p}g_{p}$ is an element of $\GL_{3}(\adele_{f})$, we define
$$g.L=\{ \alpha \in \Q^3: \iota_p(\alpha) \in {g_{p}.L_{p}} \mbox{ for all $p$} \},$$
where $\iota_p: \Q^3 \hookrightarrow \Q_p^3$ is the inclusion. 
This action is transitive, and the stabilizer of $L_{0}=\Zz^3$ under this action is $\GL_3({\whZ})$.
 
 \subsection{}\label{sec:genus}
 The $\SO_{3}$-{\em genus} of the lattice $L_{0}$ is, by definition, the orbit of $L_{0}$ under the action of   $\SO_{3}(\Aaf) \subset \GL_3(\adele_f)$: in other words, this is the set of all rational lattices which are everywhere locally isometric to $L_{0}$: 
 $$\gen_{\SO_{3}}(L_{0})=\SO_{3}(\Aaf).L_{0}\simeq \SO_{3}(\Aaf)/\SO_3(\whZ). $$
 The set of {\em genus classes} of $L_{0}$ is the set of $\SO_{3}(\Qq)$-orbits in $\gen_{\SO_{3}}(L_{0})$.
 
 \begin{prop*}
There is only one $\SO_{3}$-genus class -- that is, the quadratic form $x^2+y^2+z^2$ has genus one.  Equivalently,
 \begin{equation}
\SO_3(\Aaf) = \SO_3(\Q)\SO_3(\whZ).
\label{eq:genus1}
\end{equation}
\end{prop*}
\proof Let $L\in\gen_{\SO_{3}}(\Aaf).L_{0}$, in particular the covolume of $L$ equals the covolume of $L_{0}$ which is one. We identify $L_{0}$ with the traceless integral quaternions $\B^{(0)}(\Zz)$ and $\SO_{3}$ with $\PB$; in these terms, we need to show that there is $q\in\B^\times(\Qq)$ such that $qLq^{-1}=\B^{(0)}(\Zz)$. By definition, there is $q_{f}\in\B^\times(\Aaf)$ such that $L=q_{f}\B^{(0)}(\Zz)q_{f}^{-1}$ and so $\OO:=\Zz+L=q_{f}(\Zz+\B^{(0)}(\Zz))q_{f}^{-1}$ is a lattice in $\B(\Qq)$ containing the identity and stable by multiplication, i.e. an order; hence, by \S \ref{integralstructure}, there is $q\in\B^\times(\Qq)$ such that
$q\OO q^{-1}\subset\B(\Zz)$ and so $qLq^{-1}\subset\B^{(0)}(\Zz)$; since $qLq^{-1}$ and $\B^{(0)}(\Zz)$ have the same covolume they are equal.
\qed

If we replace $x^2+y^2+z^2$  by a different ternary quadratic form $Q$, this will in general not be so. 

\subsection{Level structure}\label{levelstructure} By construction, the quadratic form $x^2+y^2+z^2$ takes integral values on any lattice in
$\gen_{\SO_{3}}(L_{0})$. In particular, given $q\geq 1$ an integer, the quotient lattice $L/qL\simeq (\Zz/q\Zz)^3$ is naturally a quadratic space for the form $x^2+y^2+z^2$. {\em A $q$-level structure} on such a lattice is an additional datum related to 
$L/qL$. Here, we will consider two type of level $q$-structures:
\begin{itemize}
\item[-] the principal $q$-structure: this is the datum of an isomorphism of quadratic spaces $\theta: L/qL\simeq (\Zz/q\Zz)^3$. We will use this only for $q=3$ and mainly for cosmetic purposes.
\item[-] a weak $q$-structure: this is the datum of a point $\obfx\in L/qL$ such that $\obfx.\obfx\equiv d(\modu q)$ when $(d,q)=1$. This will be the main structure considered in the present paper.
\end{itemize}

Related to these level structures, are the open compact subgroups of $\SO_3(\whZ)$,
$$K_{f}[q],\ K'_{f}[q],\ K_{f}[3,q]$$ defined by
\refs{Kfdef}, \refs{K'fdef}, \refs{Kfaqdef} relative to the choice of some base point $\obfx_{q} \in \Hd(q)$.

\subsection{$\tildeHd(q)$ as an adelic quotient}

Let us consider first the set of pairs
$$\mcP_{(q)}:=\{(L,\obfx),\ L\in\gen_{\SO_{3}}(L_{0}),\ \obfx\in L/qL,\ \obfx.\obfx\equiv d(\modu q)\}.$$
The group $\SO_{3}(\Qq)$ acts diagonally on $\mcP_{(q)}$. From \eqref{eq:genus1} any $\SO_{3}(\Qq)$-orbit in $\mcP_{(q)}$ contains a pair of the form $(L_{0},\obfx)$ with $\obfx\in\H_{d}(q)$. Moreover if two pairs $(L_{0},\obfx)$
and $(L_{0},\obfx')$ give rise to the same $\SO_{3}(\Qq)$-orbit then $\obfx$ and $\obfx'$ differ by an element of 
$\SO_{3}(\Qq)\cap\SO_{3}(\whZ)=\SO_{3}(\Zz)$. It follows that the map
$\obfx\in\Hd(q)\mapsto (L_{0},\obfx)\in\mcP_{(q)}$ induces a bijection:
$$\SO_3(\Z) \bash \Hd(q) \stackrel{\sim}{\rightarrow}  \SO_3(\Qq) \backslash \mcP_{(q)}.$$
 In the sequel, we will denote by $[L,\obfx]$ the $\SO_{3}(\Qq)$-orbit of the pair $(L,\obfx)$.
  
In fact, since $L/qL\simeq (L\otimes_{\Zz}\whZ)/q(L\otimes_{\Zz}\whZ)$, the whole group $\SO_{3}(\A_{f})$ acts diagonally on $\mcP_{(q)}$. 
\begin{lem*}This action is transitive: fixing $\obfx_{q}\in\Hd(q)$, we have
$$\mcP_{(q)}=\SO_{3}(\A_{f}).(L_{0},\obfx_{q}).$$
\end{lem*}
\proof
Indeed any $(L,\obfx)\in\mcP_{(q)}$ is in the orbit of a pair of the form $(L_{0},\obfx')$ for some $\obfx'\in\Hd(q)$. It follows from the
Lemma below that $\SO_{3}(\whZ)$ (through its projection to $\SO_{3}(\Zz/q\Zz)$) acts transitively on
$\Hd(q)$. Taking $k_{\bfx'} \in \SO_{3}(\whZ)$ such that $k_{\obfx'}\obfx_{q}=\obfx'$, we have $k_{\obfx'}^{-1}(L_{0},\obfx')=(L_{0},\obfx_{q})$.\qed

\begin{lem} \label{transitive} For any prime $p$,
the group $\SO_3(\Z_p)$ acts transitively on $\Hd(\Z_p)$ (defined in the vident way). Consequently, $\SO_{3}(\widehat\Zz)$ acts transitively on 
$\Hd(\widehat\Zz)$. 
\end{lem}
\proof This -- which
can be thought of as an analogue of Witt's theorem for rank $3$ lattices over $\Z_p$ -- is immediate from Proposition~\ref{pr:localaction}, using the first case if $p=2$ and the second case if $p$ is odd.\qed

We thus have
$\mcP_{(q)}\simeq \SO_{3}(\A_{f})/K'_{f}[q].$ From the above discussion, we deduce that
 \be\label{adelictildeHdq}
\tildeHd(q) \stackrel{\sim}{\rightarrow}  \SO_3(\Qq) \backslash \mcP_{(q)} \stackrel{\sim}{\leftarrow} \SO_3(\Q) \bash \SO_3(\Aaf) / K'_f[q].
\ee

 \subsubsection{Lifting to $\Hd(q)$}\label{liftingHdq}This adelic realization of $\tildeHd(q)$ may, in fact, be lifted to an adelic realization of $\Hd(q)$ itself, at least when $q$ is coprime with $3$. For this we add the additional data of the principal level $3$-structure discussed in \S \ref{levelstructure}. This is a little bit artificial, relying on the special fact that the reduction modulo $3$ maps yields an isomorphism $\SO_{3}(\Zz)\simeq\SO_{3}(\Z/3\Z)$. 

Consider the set of triples
\begin{multline*}\mcP_{(3,q)}:=\\
\{(L,\obfx,\theta),\ L\in\gen_{\SO_{3}}(L_{0}),\ \obfx\in L/qL,\ \obfx.\obfx\equiv d(q),\ \theta:L/3L\simeq (\Zz/3\Zz)^3\}.
\end{multline*}
As above, $\SO_{3}(\A_{f})$ (hence its subgroup $\SO_{3}(\Qq)$) acts diagonally on $\mcP_{(3,q)}$: explicitly for $g\in\SO_{3}(\A_{f})$, we have
$$g.(L,\obfx,\theta)=(g.L,g.\obfx,\theta\circ g^{-1}).$$
 We consider first the $\SO_{3}(\Qq)$-orbits in $\mcP_{(3,q)}$. From \eqref{eq:genus1} any $\SO_{3}(\Qq)$-orbit contains a triple of the form $(L_{0},\obfx,\theta)$. Moreover, since reduction modulo $3$ is an isomorphism $\SO_{3}(\Zz)\simeq\SO_{3}(\Z/3\Z)$, the map $\obfx\mapsto (\obfx,\mathrm{Id})$ yields a bijection between
$\Hd(q)$ and the set of $\SO_{3}(\Zz)$-orbits of pairs $\{(\obfx,\theta),\  \obfx\in\Hd(q),\ \theta:(\Zz/3\Zz)^3\simeq (\Zz/3\Zz)^3\}$. From this, we deduce that the map
$\obfx\in\Hd(q)\mapsto (L_{0},\obfx,\mathrm{Id})\in\mcP_{(3,q)}$ induces a bijection:
$$
\Hd(q)\stackrel{\sim}{\rightarrow}  \SO_3(\Qq) \backslash \mcP_{(3,q)}. $$

Returning to the action of the whole group $\SO_{3}(\A_{f})$, we have
\begin{lem*}This action is transitive: fixing $\obfx_{q}\in\Hd(q)$, we have
$$\mcP_{(3,q)}=\SO_{3}(\A_{f}).(L_{0},\obfx_{q},\mathrm{Id}).$$
\end{lem*}
\proof The proof is exactly as above: any triple is in the orbit of a triple of the form $(L_{0},\obfx,\theta)$, $\obfx\in\Hd(q)$, 
$\theta:(\Zz/3\Zz)^3\simeq (\Zz/3\Zz)^3$. This follows from the fact that $\prod_{p|q}\SO_{3}(\Zz_{p})\times\SO_{3}(\Zz_{3})$ acts transitively on the set of pairs $(\obfx,\theta)$; recall (\S \ref{notn}) that
we are assuming that $q$ is prime to $3$. 
\qed

We have $\mcP_{(3,q)}\simeq \SO_{3}(\A_{f})/K_{f}[3,q]$. (Recall the notation from \S \ref{levelstructure}). From this and the above discussion, it follows that the map
$$\obfx\in\Hd(q)\mapsto (L_{0},\obfx,\mathrm{Id})\in\mcP_{(3,q)}$$ induces a bijective map:
 \be\label{adelicHdq}
\Hd(q)\stackrel{\sim}{\rightarrow}  \SO_3(\Qq) \backslash \mcP_{(3,q)} \stackrel{\sim}{\leftarrow} \SO_3(\Q) \bash \SO_3(\Aaf) / K_f[3,q].
\ee

\subsection{The sphere as an adelic quotient}

For completeness, we recall the interpretation of sphere $S^2$ as an adelic quotient;
 this will not be used there but it is the way to proceed in order to adapt the proof of Theorems \ref{linnikQ} and \ref{mvQ} to obtain Theorems \ref{linnikinfty} and $\ref{MVinfty}$ or to obtain hybrid equidistribution theorems. Let $\Aa=\Rr\times\Aaf$ denote the full ring of ad\`eles, then
  $\SO_{3}(\Aa)=\SO_{3}(\Rr)\times\SO_{3}(\Aaf)$; hence, by \eqref{eq:genus1} we have the identification
  $$\SO_{3}(\Zz)\bash \SO_{3}(\Rr)\simeq \SO_{3}(\Qq)\bash\SO_{3}(\Aa)/\SO_{3}(\whZ).$$
 Since $\SO_{3}(\Rr)$ acts transitively on $S^2$, we obtain -- choosing some point $x_{\infty}\in S^2$ and  letting
  $$K_{\infty}=\SO_{x_{\infty}}(\Rr)\simeq \SO_{2}(\Rr)$$
--  the identification
 \be\label{adelictildesphere}\SO_{3}(\Zz)\bash S^2\simeq \SO_{3}(\Qq)\bash\SO_{3}(\Aa)/K_{\infty}.\SO_{3}(\whZ).
 \ee
 As in the previous section, we may remove the quotient by $\SO_{3}(\Zz)$ by adding the principal $3$-structure and obtain
 $$\SO_{3}(\Rr)\simeq \SO_{3}(\Qq)\bash\SO_{3}(\Aa)/K_{f}[3].$$
\be\label{adelicsphere}S^2\simeq \SO_{3}(\Qq)\bash\SO_{3}(\Aa)/K_{\infty}.K_{f}[3].
\ee

\section{Adelic interpretation of $\Hd$.} \label{sec:adelicHd}

In this section, we describe the first line of the diagrams \eqref{massive-identificationtilde} and \eqref{massive-identification}. 

The significance of this in the proof is as follows: 
We previously described an action of $\Pic(\OO_K)$ on $\tildeHd$,
and also a {\em lifting} of this action -- at least in the case of a prime ideal above $5$ -- to
$\mathscr{H}_d$.  After the present section, we will understand 
this action (and also the lifted action, although this is only of cosmetic interest)
in terms of ad\`eles. This will be a key tool in the proof of 
Proposition~\ref{Naddles} and Proposition~\ref{JLD}. 

 The presentation of this section
follows that in \cite{EV}, but we emphasize that the material is in essence classical. 

\subsection{}\label{classgroupadelicsection}

As before, let $\P$ be the set of pairs $$\P =\{ (L,\x),\ \x.\x = d,\ \x \in L,\ L \in \gen_{\SO_{3}}(L_{0})\}.$$
By contrast with $\mcP_{(q)}$ or $\mcP_{(3,q)}$, the set $\mcP$ carries no natural action of $\SO_{3}(\A_{f})$; on the other hand,
the group $\SO_3(\Q)$ acts on $\P$ diagonally. By \eqref{eq:genus1}, every $\SO_3(\Q)$-orbit in $\mcP$ contains an element of the form $(L_0,\x)$ (where $\bfx$ belongs to $\Hd$); two such pairs differ by the action of a unique element of $\SO_{3}(\Q)\cap\SO_{3}(\whZ)=\SO_3(\Z)$.  From this is follows that the map $\bfx\in\Hd\mapsto (L_{0},\bfx)\in\mcP$ induces a bijective map
\begin{equation}
\tildeHd \stackrel{\sim}{\rightarrow} \SO_3(\Q)\bash\P.
\label{eq:so3qp}
\end{equation}

We will now realize $\SO_3(\Q)\bash\P$ as an adelic quotient by endowing it with a transitive action of a {\em subgroup} of 
$\SO_{3}(\A_{f})$.

Choose an element $\x_0 \in \Hd$, and let $\SO_{\x_0}$ be its stabilizer in $\SO_3$. 
 Since the quadratic space is $3$-dimensional, $\SO_{\x_0}$ is the special orthogonal group of a quadratic plane (namely $\Qq\x_{0}^\perp$),  and so is a $1$-dimensional torus.

The group $\SO_{\x_{0}}(\A_{f})$ acts (by multiplication on the first coordinate) on the subset of $\mcP$ consisting of pairs of the form $(L,\x_{0})$.  We can extend this action to the whole of $\SO_3(\Q)\bash\P$ as follows.  By Witt's theorem, every $\SO_3(\Q)$-orbit in $\P$ is of the form $[L,\x_0]$, for some $L \in \SO_3(\A_f)L_{0}$.  
Then or any $t\in\SO_{\x_0}(\A_{f})$, we define $$t.[L,\x_{0}]:=[t.L,\x_{0}].$$ This is well defined: if 
$[L,\x_{0}]=[L',\x_{0}]$ then $L$ and $L'$ differ by an element
 $\tau\in\SO_{x_{0}}(\Qq)$; because of the commutativity of $\SO_{\x_{0}}$, $$[t\tau L,\x_{0}]=[\tau tL,\x_{0}]=[tL,\x_{0}].$$
 This action is transitive: consider the orbit of $(L,\x_0)$ where $L=g.L_{0}$ and $g\in\SO_{3}(\A_{f})$. Since $L$ contains $\x_0$,  
 $L_0\otimes \adele_f$ contains $g^{-1} \x_0$; by Lemma~\ref{transitive} there is an element $k$ of $\SO_3(\whZ)$ which sends $\x_0$ to $g^{-1} \x_0$.  In particular, $g k=t$ lies in $\SO_{\x_0}(\A_f)$ and 
 $t.[L_0,\x_{0}] =[gk L_0,\x_{0}]= [L,\x_{0}]$. 
 
 Finally, from the definition of the action it is easy to see that the stabilizer of $[ L_0,\x_{0}]$ is the subgroup $\SO_{\x_0}(\Q)\SO_{\x_0}(\whZ)$ where $\SO_{\x_0}(\whZ)=\SO_{\x_0}(\A_{f})\cap\SO_{3}(\whZ)$.  
Hence \eqref{eq:so3qp} extends to bijections
\begin{equation}
\SO_3(\Z)\bash\Hd \stackrel{\sim}{\rightarrow} \SO_3(\Q)\bash\P \stackrel{\sim}{\leftarrow}
\SO_{\x_0}(\Q) \bash \SO_{\x_0}(\Aaf) / \SO_{\x_0}(\whZ). 
\label{adelictildeHd}
\end{equation}

 \begin{rem*} This setup is very specific to the three dimensional situation, in particular the fact that $\SO_{x_{0}}$ is commutative.   If one studies the question of the representations of an integer $d$ by a higher rank quadratic form $Q$, the quotient $\SO_{Q}(\Zz)\bash\Hd$, when non-empty, still has a description in terms of double cosets of an adelic group (see \cite{EV} for details) but will not carry an {\em action} of an adelic group as it does here.
\end{rem*}
 \subsection{}
We shall now identify $\SO_{\x_0}(\Q) \bash  \SO_{\bfx_0}(\Aaf) / \SO_{\bfx_0}(\whZ)$  with
a quotient of the commutative group $\Pic(\OO_K)$. Therefore, the identification \eqref{adelictildeHd}
may be considered as giving an action of $\Pic(\OO_K)$ on $\tildeHd$. 

For this purpose, it is again most convenient to phrase everything in terms of quaternions via the identifications $\Q^3\simeq\B^{(0)}(\Qq)$ and $\SO_{3}\simeq \PB$.  In particular we view $\x_0$ as a trace-free quaternion.  Now the stabilizer $\SO_{\x_{0}}\simeq\PB_{\x_{0}}$ is the multiplicative group generated by (conjugations by) invertible quaternions of the form $a+b\x_{0}$.  In the previous section, we have defined a transitive action of $\SO_{\bfx_{0}}(\Aaf)\simeq\PB_{\bfx_{0}}(\A_{f})$ on $\tildeHd$. 

Now, the map $a+b\sqrt{-d}\mapsto a+b\x_{0}$ defines via conjugation %
 an isomorphism of $\Qq$-algebraic groups  
\be\label{resKisom}\mathrm{Res}_{K/\Q} \G_m/\G_m\simeq\PB_{\x_{0}};
\ee
 in particular, for any field $F\supset\Qq$,
$\PB_{\x_0}(F) \simeq (F \tensor_\Q K)^{\times} / F^{\times}$.

Thus, denoting by 
 $\adele^{\times}_{K,f}=(\adele_f \otimes K)^{\times}$ the group of finite id\`eles of $K$, \refs{resKisom} defines an action of $\adele^{\times}_{K,f}$ on $\tildeHd$. The subgroups $\A_{f}^\times$ and $K^\times$ act trivially, as does
 $U \subset \adele^{\times}_{K,f}$, the preimage of $\PB_{\x_{0}}(\whZ)$ under \eqref{resKisom}. The maximal compact subgroup $\widehat{\OO}_{K}^\times=(\OO_{K}\otimes_{\Z}\whZ)^\times$ of $\adele^{\times}_{K,f}$ is certainly contained in $U$; therefore, we have a well defined transitive action of 
 \beq
K^{\times} \bash \adele^{\times}_{K,f} / \widehat{\OO}^\times_{K} = \Pic(\OO_K)
\eeq
on $\tildeHd$. This action being transitive, $\tildeHd$ is identified with a quotient of $\Pic(\OO_K)$,
and is in fact a torsor under  the abelian group $
K^{\times} \bash \adele^{\times}_{K,f} /  U$.
A careful study of the inclusion $\widehat{\OO}^\times_{K}\subset U$ at $2$ recovers Proposition~\ref{torsor1}. 

{\em A priori}, the action we have defined depends on $\bfx_0$. We verify independence in \S \ref{sec:indep}.

\subsection{Lifting the action to $\Hd$}\label{liftingactionHd}
\label{removingso3}   As before,  by introducing an extra level $3$-structure, we may replace $\tildeHd$ by its covering $\Hd$ and obtain a lift of the action of $\SO_{\bfx_{0}}(\A_{f})\simeq\PB_{\bfx_{0}}(\A_{f})$ and thus of $\A^\times_{K,f}$ on $\tildeHd$ to an action on $\Hd$. Notice that this latter action is, in general,
not necessarily transitive; nor does it, necessarily, factor through the class group,
but only through a certain ray class group. 

Consider the set of triples
$$\Pq = \{ (L,\x, \theta ), \ L \in \gen_{\SO_{3}}(L_0),\ \x \in L,\ \x.\x = d, \theta: L/3L \simeq (\Z/3\Z)^3\}.$$ 
Using that the reduction mod $3$ map from $\SO_3(\Z)$ to $\SO_3(\Z/3\Z)$ is an isomorphism, we may verify as above that the map
$$\mathscr{H}_d \rightarrow \SO_{3}(\Qq)\bash\Pq, \ \ \x \mapsto [L_{0},\x, \mathrm{Id}]$$
is bijective.

As above, the group $\SO_{\bfx_{0}}(\A_{f})$ acts on $\SO_{3}(\Qq)\bash\Pq$: by Witt's Theorem, every $\SO_3(\Q)$-orbit in $\Pq$ is of the form  $[L, \x_{0}, \theta]$  and we set for $t\in\SO_{\bfx_{0}}(\A_{f})$,
$$t.[L, \x_{0}, \theta]=[t.L,\bfx_{0},\theta\circ t^{-1}];$$
from the previous discussion, this action is well defined. This action, however, is {\em not} transitive: 
 the various $\SO_{\bfx_{0}}(\A_{f})$-orbits are parametrized by the orbits of $\SO_{3}(\Z/3\Z)$ under the action of $\SO_{\bfx_{0}}(\whZ)$. In other words, we have 
a bijection 
\begin{equation} \label{adelicHd} \Hd \stackrel{\sim}{\rightarrow}  \SO_3(\Q)\bash\Pq  \stackrel{\sim}{\leftarrow} \SO_{\x_0}(\Q) \bash\SO_{\x_0}(\adele_f)\! \times \SO_3(\Z/3\Z)\! / \SO_{\x_0}(\whZ),\end{equation}
where $\SO_{\x_0}(\whZ)$ acts diagonally on the product $ \SO_{\x_0}(\adele_f) \times \SO_3(\Z/3\Z).$
Under this identification, the action of $t\in\SO_{\x_0}(\adele_f)$ is the one induced by
$$t.[t',\kappa]=[tt',\kappa],\ (t',\kappa)\in \SO_{\x_0}(\adele_f)\!  \times\!  \SO_3(\Z/3\Z).$$ 
Thus \eqref{adelicHd} gives us the desired way to lift the action of $\A^\times_{K,f}$ to $\Hd$. 

\subsection{Independence w.r.t. $\bfx_{0}$}\label{sec:indep}
Let us see that the above defined actions of $\A^\times_{K,f}$ on $\Hd$, $\tildeHd$ do not depend on the choice of the base point $\bfx_{0}$.

 Let $\bfx'_{0}\in\Hd$ be another point. By Witt's theorem 
$\x'_{0}=\rho.\x_{0}$, for some 
$\rho \in \PB(\Q)$. Then $\PB_{\x_{0}'}=\rho\PB_{\x_{0}}\rho^{-1}$. Let $u=a+b\sqrt{-d}$ be an element of $\A^\times_{K}$ ($a,b\in\A_{f}$), and let $t_{u}$ (resp. $t'_{u}$) denote the corresponding element in $\PB_{\x_{0}}(\A_{f})$ (resp. in $\PB_{\x'_{0}}(\A_{f})$).  Then $t'_{u}=\rho t_{u} \rho^{-1}.$ It will suffice to see that $t_{u}.\bfx_{0}=t'_{u}.\bfx_{0}$, or equivalently
$$t_{u}[L_{0},\bfx_{0},\Id]=t'_{u}[L_{0},\bfx_{0},\Id];$$
the latter expression equals \begin{align*}
t'_{u}[L_{0},\rho^{-1}\bfx'_{0},\Id]&=t'_{u}[\rho L_{0},\bfx'_{0},\rho^{-1}]=[t'_{u}\rho L_{0},\bfx'_{0},\rho^{-1}{t'}_{u}^{-1}]\\
&=[\rho^{-1}t'_{u}\rho L_{0}, \rho^{-1}\bfx'_{0},\rho^{-1}{t'}_{u}^{-1}\rho]\\
&=[t_{u}L_{0},\bfx_{0},t_{u}^{-1}]=t_{u}[L_{0},\bfx_{0},\Id].
\end{align*}
\qed

\section{Adelic interpretation of $\rmr_{q}:\Hd \rightarrow \Hd(q)$.}\label{sec:adelicredq}

In this section, we interpret the reduction maps, $\rmr_{q}$ and $\rmr_{\infty}$ in terms of the adelic quotients from the previous sections.

 Recall that, in \refs{adelictildesphere}, \refs{adelictildeHdq}, \refs{adelictildeHd} (resp. \refs{adelicsphere}, \refs{adelicHdq}, \refs{adelicHd}) we have given adelic identifications of $\widetilde S^2,\ \tildeHdq$ and $\tildeHd$, and of their finite coverings $S^2,\ \Hdq$ and $\Hd$. 
\subsection{}
First, let us recall the identifications
$$\tildeHd\simeq \SO_{3}(\Qq)\bash\mcP,\ \tildeHdq\simeq \SO_{3}(\Qq)\bash\mcP_{(q)}$$ where
$$\SO_{3}(\Qq)\bash\mcP=\{[L,\bfx], L\in\SO_{3}(\A_{f}).L_{0},\ \bfx\in L,\ \bfx.\bfx=d\}$$
$$\SO_{3}(\Qq)\bash\mcP_{(q)}=\{[L,\obfx],\ L\in\SO_{3}(\A_{f}).L_{0},\ \obfx\in L/qL,\ \obfx.\obfx\equiv d(q)\}.$$

The map $\rmr _{q}:\tildeHd \ra \tildeHdq$ is induced by the natural map
$$\bfx\in L\mapsto \obfx\in L/qL$$
which we also denote $\rmr_{q}$.

We now explain how to write $\rmr_q$ in the adelic language.  Let $t$ be an element of $\SO_{\x_0}(\Aaf)$ and $[t]$ its double coset in $\SO_{\x_{0}}(\Qq)\bash\SO_{\x_{0}}(\Aaf)/\SO_{\x_{0}}(\whZ)$.  Recall that $\x_0$ and $\obfx_q$ were basepoints in $\tildeHd$ and $\tildeHdq$ respectively.  We will demonstrate that the reduction map $\rmr_q$, thought of as a map
\beq
\rmr_q: \SO_{\x_{0}}(\Qq)\bash\SO_{\x_{0}}(\Aaf)/\SO_{\x_{0}}(\whZ)
\ra
\SO_{3}(\Qq)\bash\SO_{3}(\Aaf)/K_{f}[q]
\eeq
is the one sending $[t]$ to $[t.k_{\obfx_{0}}]$, where $k_{\obfx_{0}}\in\SO_{3}(\whZ)$ is a fixed element satisfying $k_{\obfx_{0}}.\obfx_{q}\equiv \x_{0}(\modu q)$.

To see this, let $t \in \SO_{\x_0}(\A_f)$ be a representative for one of the double cosets in $\SO_{\x_{0}}(\Qq)\bash\SO_{\x_{0}}(\Aaf)/\SO_{\x_{0}}(\whZ)$.  We may choose $t$ to have integral coordinates at all primes dividing $q$.  Write $\beta = \gamma k$, with $\gamma \in \SO_3(\Q)$ and $k \in \SO_3(\whZ)$.  Then by the definitions of \refs{adelictildeHd} and \refs{adelicHdq}, one finds that the element of $\SO_3(\Z) \bash \Hd$ corresponding to $t$ is $\gamma^{-1} \x_0$, while the element of $\SO_3(\Z) \bash \Hd(q)$ corresponding to $tk_{\obfx_{0}}$ is $k \bar{\x}_0$.  But $\gamma k$ fixes $\x_0$ (whence also $\bar{\x}_0$), so the reduction of $\gamma^{-1} \x_0$ is precisely $k \bar{\x}_0$.

Again, we can lift the situation to $\Hd$: there is a natural map
$$\SO_{\x_{0}}(\Qq)\bash\SO_{\x_{0}}(\Aaf)\! \times\!  \SO(\Z/3\Z) /\SO_{\x_{0}}(\whZ)\mapsto\SO_{3}(\Qq)\bash\SO_{3}(\Aaf)/K_{f}[3,q]$$
which corresponds to the reduction map $\Hd \rightarrow \Hd(q)$.  Explicitly, choosing for base points the triples $(L_{0},\bfx_{0},\Id)$, $(L_{0},\obfx_{q},\Id)$ the map is given by
\be\label{adelicredq}[t,\kappa]\mapsto [tk_{\obfx_{0}}k_{3}]\ee
 where $k_{3}\in\SO_{3}(\Zz_{3})$ is a lift of 
$\kappa\in \SO_{3}(\Zz/3\Zz)$ and $k_{\obfx_{0}}$ is as above but such that its component at the place $3$ is trivial.

\part{Graphs and Expanders }

\section{The graph structure on $\Hd(q)$}\label{sec:Hdqexpander}

We shall now replace the role of the finite ad\`eles $\adele_f$ in the bijection \refs{adelicHdq}
by the much smaller ring $\Q_5$. More precisely, we will show the existence of a bijection:
\begin{equation} \label{treebijections} \Hd(q) \stackrel{\sim}{\leftarrow}   \Gamma_{(3,q)} \bash \SO_3(\Q_5) / K_5
,\end{equation} 
where $\Gamma_{(3,q)}\subset \SO_3(\Q_5)$ is a suitable lattice (i.e. a discrete cofinite subgroup) and $K_5$ is the maximal compact subgroup $\SO_3(\Z_5)$.  We have 
$$\SO_3(\Q_{5})\simeq\PB(\Q_{5})\simeq \PGL_{2}(\Q_{5})$$
(since $\B(\Qq_{5})\simeq M_{2}(\Qq_{5})$), therefore the quotient
$\SO_3(\Q_5) / K_5$ is identified with $$\PGL_{2}(\Q_{5})/\PGL_{2}(\Z_{5})=:\mcT_{5}$$
 which has the structure of an infinite $6$-valent tree (namely,  the {\em Bruhat-Tits} tree of $\PGL_{2}(\Qq_{5})$, cf. \cite{Serretree}). 
 
 The set $\Hd(q)$ has thus a structure of  a finite quotient of $\mcT_{5}$ and we will see that this graph structure coincides with that described in \S \ref{HdqExpanderGraph}.  In particular, the latter is connected. 
 
From this viewpoint, it will be possible to prove Proposition \ref{Naddles} and Proposition \ref{JLD} (i.e. the graph $\Hdq$ is an expander). The latter relies on the theory of automorphic forms, especially the Jacquet-Langlands correspondence and the work of Eichler-Shimura. 

Let us mention that a good part of this section is closely related to the book of Lubotzky \cite{Lubot}, especially Chapter 6 and the Appendix, in which the reader will find a motivated discussion of the passage between automorphic forms and expanders.

\subsection{$\Hd(q)$ as a quotient of a tree}\label{graphadelic}

 If $w$ is an element of $\PB(\Q_5)$, we denote by $[w]_5$ the element of $\PB(\adele_f)$ which projects to $w$ at the $5$-adic place, and to  the identity everywhere else.  When no confusion is likely (as in the statement of the following lemma) we will identify $\PB(\Q_5)$ with the subgroup $[\PB(\Q_5)]_5$ of $\PB(\Aaf)$.

\begin{lem*}
\label {le:approximation} One has
\beq
\PB(\Qq).\PB(\Qq_{5}).K_f[3,q]=\PB(\Aaf).
\eeq
Consequently, the map $g\in\PB(\Q_{5})\ra [g]_{5}\in\PB(\A_{f})$ yields a bijective map
\be\label{adelicHdqtree}\Gamma_{(3,q)}\bash \PB(\Qq_{5})/K_{5}\stackrel{\sim}{\ra} \PB(\Q)\bash\PB(\Aaf)/K_{f}[3,q]\stackrel{\sim}{\leftarrow}\Hdq\ee
where $\Gamma_{(3,q)}$ is the lattice $\PB(\Qq)\cap \PB(\Q_{5}).K_{f}[3,q]$.
\end{lem*}

\begin{proof} This main ingredient of the proof is the so-called {\em strong approximation property} (for simply connected semi-simple algebraic groups): we will not discuss this property in any generality and refer to \cite[Chap. 7]{PR} for a complete treatment. Alternatively, the reader may also refer to \cite[Chap. III ]{Vigneras} for a discussion of the strong approximation property in the context of quaternion algebras. For now, let us merely say that, if $\PB$ satisfied the strong approximation property the assertion would follow immediately; unfortunately, $\PB$ is not simply connected.

One remedies this problem by passing from the $\Qq$-algebraic group to its double cover $\B^{1}$ (cf. \refs{Qcovering}). The group $\B^{1}$ is simply connected and semisimple. Moreover, $\B^{1}(\Q_5)$ is noncompact.   Thus, it satisfies a strong approximation property:  for any open compact $\Omega\subset \B^{1}(\Aaf)$, we have 
\beq
\B^{1}(\Qq).\B^{1}(\Qq_{5}).\Omega=\B^{1}(\Aaf).
\eeq

It follows that $\PB(\Qq).\PB(\Qq_{5}).K_f[3,q]$
contains the image $$\Theta=\B^{1}(\adele_f)/\{\pm 1\}$$ of $\B^{1}(\adele_f)$ in $\PB(\adele_f)$. 
It will suffice, then, to verify that 
\begin{equation} \label{Theta-equality} \PB(\Q). \Theta. K_{f}[3,q] = \PB(\adele_f); \end{equation}
this will follow from \eqref{eq:genus1} if $(\Theta \cap \PB(\whZ)) . K_f[3,q] = \PB(\whZ)$; equivalently,
if $\Theta \cap \PB(\whZ)$ acts transitively on $\PB(\whZ)/K_{f}[3,q]\simeq \Hd(q)$.

In turn, it is equivalent to show that $\Theta_{p} \cap \PB(\Z_p)$ acts transitively on $\Hd(\Z_p)$ for each $p | 3q$, where
$\Theta_{p}$ is the image of $\B^{1}(\Qq_{p})$ in $\PB(\Q_p)$. 

Recall that $\Hd(\Zz_{p})$ is identified with $\B^{(0,d)}(\Zz_{p})$, the trace $0$ quaternion of norm $d\in\Zz_{p}^\times$; then $\Theta_{p} \cap \PB(\Z_p)$ contains $\B^{1}(\Zz_{p})/\{\pm 1\}$ acting by conjugation.  The transitivity of this action follows from the second part of Proposition~\ref{pr:localaction},  using the fact that $(p,2d)=1$.
 \end{proof}

 \subsection{The graph structure}\label{graphstructure}
We can describe the (Bruhat-Tits) graph structure on $\PB(\Q_5)/\PB(\Zz_5)$ in two equivalent ways. 
We write $K_5 = \PB(\Zz_5)$ in what follows.  Recall that $\mcA_5$ consists of a set of six matrices (as in \S \ref{primeaction}) or six quaternions (as in \S \ref{ss:explicitaction}) depending on whether we are working in $\SO_{3}$ or $\PB$.

 \begin{enumerate} \item   Fix $\alpha_{5}\in\B^\times(\Q_{5})$ such that the $5$-adic valuation of its norm is $\pm 1$.  Write $\overline{\alpha_5} $ for its image in $\PB(\Q_{5})$; then
we have \be\label{5adicrepresentative} K_5.\overline{\alpha_{5}}.K_5=\bigsqcup_{w \in \mcA_5}[w]_5 K_{5}.\ee
 We join any coset $gK_5$ to the six cosets $g [w]_{5}K_{5}: w \in \mcA_5 $.  The resulting structure is independent of $\alpha_5$. 
 \item More intrinsically, we may identify the quotient  $\SO_{3}(\Q_5)/\SO_{3}(\Zz_5)$
 with the sublattices of $\Q_5^3$ in the orbit of $\SO_{3}(\Q_5) . \Zz_5^3$.
 Given any such sublattice $L$, the quadratic form takes values on $\Zz_{5}$ on that lattice and the induced quadratic form on $L/5L \cong (\Z/5\Z)^3$ takes values
 in $\Z/5\Z$; there are precisely six ($= | \mathbb{P}^1(\Z/5\Z)|$)  isotropic lines. 
 Choosing any $\mathbf{v} \in L$ so that the image of $\mathbf{v}$ in $L/5L$ isotropic (but non-zero),
 we may construct the new lattice:
 $$L' = \langle \mathbf{v}/5 \rangle + \{ \mathbf{w} \in L: \mathbf{v} . \mathbf{w} \equiv 0 \mbox{ mod $5$}.\}$$
 Then $L'$ depends only on the line spanned by $\mathbf{v}$ in $L/5L$, and belongs
 to $\SO_{3}(\Q_5) .  \Zz_5^3$ also. In particular, 
 we construct six such $L'$, which we declare to be the neighbours of $L$. 
 \end{enumerate} 

Then these give equivalent graph structures and moreover:
\begin{equation} \mbox{The resulting graph is a tree, i.e., has no cycles.} \end{equation}

Later on we shall use the following property the verification of which we leave to the reader: if $L$ is a sublattice of $\Q_5^3$ representing an element of $\SO_{3}(\Q_5)/\SO_{3}(\Zz_5)$, and $\mathbf{v} \in L$ is such that
$\mathbf{v} . \mathbf{v}$ is not zero mod $5$, then:
\begin{equation} \label{geod} \mbox{The orbit $\SO_{\mathbf{v}}(\Q_5) . L$ is an infinite geodesic in the tree,} \end{equation}
(i.e. an isometric  embedding of $\Zz$ in the tree w.r.t. the obvious metrics); this is a simple consequence of the second definition.

 Let us now check that the neighbors of $\obfx\in\Hdq$ for the graph structure  defined in \S \ref{HdqExpanderGraph} correspond to the neighbors under $T_{5}$ of the image of $\obfx$ in \refs{adelicHdq}. Thus, the graph structure on $\Hd(q)$ coincides with the graph structure on $\Gamma_{(3,q)}\bash\mcT_{5}$.

 In the notations of \S \ref{liftingHdq}, the neighbors of $\obfx$ i.e. $\{w.\obfx,\ w\in\mcA_{5}\}$ correspond to the orbits of the triples $$\{[L_{0},w.\obfx,\Id],\ w\in\mcA_{5}\}=\{[w^{-1}L_{0},\obfx,\Id\circ w],\ w\in\mcA_{5}\};$$
now for any $p\not=5$ and $w\in\mcA_{5}$, $[w]_p \in K_{p}$. Moreover, $[w]_3\equiv \Id(\modu 3)$. Therefore the above set equals
$\{[w^{-1}]_{5}.[L_{0},\obfx,\Id],\ w\in\mcA_{5}\}$ which manifestly agrees with
the graph structure introduced above. 

  \subsection{Proof of Proposition \ref{Naddles}}\label{Naddlesproof} 

We focus now on the action of a prime ideal $\mfp$ above $5$ on $\tildeHd$. 
Fix  $\pi$ a uniformizer of $K_{\mfp}$, which we write in the form $a+b\sqrt{-d}\in\Z_{5}[\sqrt{-d}]$. To  any $\bfx\in\Hd$ we associate the quaternion $q_{\bfx}=a+b\bfx$ (the $5$-adic valuation of $\Nr(q_{\bfx})$ equals $1$) and the corresponding rotation $t_{\bfx}\in\SO_{\bfx}(\Q_{5})$ induced by conjugation by $q_{\bfx}$; since $K_{\mfp}^\times=\pi^\Zz\OO_{K_{\mfp}}^\times$
 we have $\SO_{\bfx}(\Q_{5})=t_{\bfx}^\Z\SO_{\bfx}(\Z_{5})$.
 
 The action of $\mfp$ can be interpreted in terms of our adelic viewpoint:
We have seen in \S \ref{liftingactionHd}, that the group $\SO_{\bfx}(\Q_{5})$ acts on $\Hd$ with $\SO_{\bfx}(\Z_{5})$ acting trivially; by projection this also defines an action on $\tildeHd$. Via the map $\pi\mapsto q_{\bfx}\mapsto t_{\bfx}$ one obtains an action of the group $\pi^\Z$ which in fact does not depend on the choice of $\bfx$ (\S \ref{sec:indep}); the action of $\pi$ coincides with the action of the ideal class $[\mfp]$ defined earlier.

Let us also recall that if $\bfx\in\Hd$ corresponds to the class $[L_{0},\bfx,\Id]\in\SO_{3}(\Q)\bash\Pq$, the element $\pi.\bfx \in \Hd$ corresponds to the class $t_{\bfx}[L_{0},\bfx,\Id]=[t_{\bfx}L_{0},\bfx,\Id \circ t_{\bfx}^{-1}] = [t_{\bfx}L_{0},\bfx,\Id]$, the final equality holding because the $3$-component of $t_{\bfx}$ is trivial. Therefore the trajectory $\pi^\Z.\bfx$ is described by the infinite sequence of lattices
$$\dots, L_{-2}, L_{-1}, L_0, L_1, L_{2},\dots,\ L_{i}=t^{i}_{\bfx}L_{0}.$$
All the $L_i$ contain $\bfx$.  Write $L_{i,p}$ for $L_{i,p}=L_{i}\otimes_{\Z}\Z_{p}$; then $L_{i,p} = L_{0,p}$ for all $i$ and all $p \neq 5$, the $p$-th component of $t_{\x}$ being trivial.  So the sequence of lattices $(L_i)$ is completely determined by the sequence $(L_{i,5})$ of its $5$-adic components. 

Since the rotation $t_{\bfx}$ comes from a quaternion whose norm has $5$-adic valuation equal to $1$, the sequence $(L_{i,5})_{i\in\Zz}$ describes an infinite geodesic passing through $L_{0,5}$ in the tree (cf. \ref{graphstructure} (1) or more generally \eqref{geod}),
$$\SO_{3}(\Q_{5}).L_{0,5}\simeq\SO_{3}(\Q_{5})/\SO_{3}(\Z_{5}).$$
Up to orientation, any such geodesic may be encoded by an infinite non-backtracking word in 
$$\mcA_{5}=\{A^{\pm 1}, B^{\pm 1}, C^{\pm 1}\},$$
where the $i$-th letter connects the $(i-1)$-st element of the geodesic along an edge of the tree to the $i$-th element. In the present case, the word associated with the sequence $(L_{i,5})_{i\in\Zz}$ is the word $(w^{-1}_{i})_{i\in\Z}$ satifying $L_{i,5}=w^{-1}_{i}L_{i-1,5}$; equivalently, $(w_{i})_{i\in\Z}$ is the word corresponding to the trajectory of $\bfx$ defined in \S \ref{primeaction}.

Suppose that two points $\x, \x' \in \Hd$ give rise to the same truncated word of length $2\ell$:
$$W^{(\ell)}: [-\ell+1, \ell] \rightarrow \{A^{\pm 1}, B^{\pm 1}, C^{\pm 1}\}=\mcA_{5}.$$
This means exactly that the geodesics $\SO_{\x}(\Qq_{5}).L_{0,5}$ and $\SO_{\x'}(\Qq_{5}).L_{0,5}$ in the Bruhat-Tits tree coincide  ``from times $-\ell$ to $\ell$ '' or in other terms
$$L_{i,5}=L'_{i,5},\ i=-\ell,\dots,\ell.$$
In particular $\bfx'\in L_{\ell,5}\cap L_{-\ell,5}$. The last intersection is a sublattice of $L_{0,5}$ of index $5^{2\ell}$; more precisely, it is the preimage in $L_{0,5}$ of the line generated by $\bfx (\modu 5^\ell)$ in $L_{0,5}/5^\ell L_{0,5}$.  Thus $\x$ and $\x'$ are linearly dependent in $L_0 / 5^\ell L_0$, and since both vectors have norm $d$,  we have $\x \equiv \pm \x'$ modulo $5^\ell$. This concludes the proof of Proposition \ref{Naddles}.  \qed

\subsection{Proof of Proposition \ref{JLD}}\label{JLDproof} 
We assume some familiarity with the theory of automorphic forms; in any case, we refer to Lubotzky's book \cite[Chap. 6 \& Appendix]{Lubot}.

 The space $L^2(\PB(\Q)\bash\PB(\A_{f})/K_{f}[3,q])$
admits an orthogonal decomposition into eigenspaces of the commutative algebra generated by Hecke operators. These eigenspaces are the set of $\PB(\Rr)K_{f}[3,q]$-invariant vectors of automorphic representations on $\PB$. Such representations are 
of two types:
\begin{itemize}
\item[-]
one-dimensional representations;
\item[-]   infinite dimensional representations. 
\end{itemize}

The latter corresponds, via the Jacquet-Langlands correspondence \cite{JL} to automorphic representations of $\PGL_{2}$ which are discrete series of weight $2$ and unramified outside $2,3$ and the primes dividing $q$ (more precisely, one can check that the conductor divides $18.q^2$). From the work of Deligne (or rather Eichler/Igusa/Shimura since this is weight $2$), the eigenvalue of the
standard $5$-Hecke operator for such spaces is bounded in absolute value by $2\sqrt 5$.

As for the former: each such is the representation of $\PB(\Aa)$ on the one-dimensional  subspace
 generated by the function
$$g\in\PB(\Q)\bash\PB(\A)\mapsto \chi( \Nr(g))$$
where $\chi:\Qq^\times\bash\Aa^\times\mapsto \{\pm 1\}$  is some quadratic character. The action of $\PB(\Qq)$ on such a representation is trivial, as is the action of $K_f[3,q]$  (by definition); moreover since the elements of $\Theta$ come from quaternions of  norm $1$, the action of  $\Theta$ is trivial as well; hence from \eqref{Theta-equality}, such representation has to be the trivial one. It follows from this enumeration that $-6$ does not occur as an eigenvalue of the $5$-th Hecke operator so the graph, while connected (cf. above) is not bipartite. 

\begin{rem*} The discussion above is also valid for $\tildeHdq$: this follows immediately from the previous discussions by projection. In particular we have
$$\Gamma_{(q)}\bash \PB(\Qq_{5})/K_{5}\stackrel{\sim}{\ra} \PB(\Q)\bash\PB(\Aaf)/K'_{f}[q]\stackrel{\sim}{\leftarrow} \tildeHdq$$
with $\Gamma_{(q)}=\PB(\Qq)\cap K'_{f}[q]\PB(\Q_{5})$.
\end{rem*}

\section{Expander graphs and random walks} \label{expander}

The contents of this section follow lecture notes of Hoori, Linial and Wigderson \cite{HLW}.
Our goal is to prove Proposition \ref{chernoff}. 

\subsection{}

Let $\mathscr{G} = (V,E)$ be a (possibly directed) $d$-regular graph on $|V| =n$ vertices, i.e.
the number of incoming edges to each vertex is $d$, and the number of outgoing edges is also $d$.  We assume $d > 2$. 
The normalized adjacency matrix $T$ of $\mathscr{G}$ acts on $L^2(V)$ by
$$T f (x) = \frac{1}{d} \sum_{(x \mapsto y) \in E } f(y). $$
By an abuse of notation, we will use $L^2(\mathscr{G})$ and $L^2(V)$ interchangeably.  More generally, $\mathscr{G}$ may be allowed to have multiple edges and loops, in which case we modify the definition of $T$ in the evident way.

Let $\|T\|$ be the operator norm of $T$ acting on the orthogonal complement of the constants
in $L^2(V)$.  The graph $\mathscr{G}$ is said to be an $\alpha$-expander, for some
$\alpha < 1$, if $\|T\| \leq 1-\alpha$. In rough terms, the smaller $\|T\|$ is, the more ``strongly connected" the graph $\mathscr{G}$. 

When we speak of a ``random walk on $\mathscr{G}$,'' we mean that we select
a vertex uniformly and randomly from $V$, and then proceed to walk along directed edges, 
at each stage choosing one of the adjacent edges one with each choice assigned probability $1/d$.

\begin{lem}
Let $Q_1, \dots, Q_{\ell}$ be subsets of $V$, with densities $\mu_i := |Q_i|/n$.  The probability that a random walk on $\mathscr{G}$ is in $Q_j$ at step $j$, for all $1 \leq j \leq \ell$, is at most
$$\prod_{i=1}^{\ell-1} \left(\sqrt{\mu_i \mu_{i+1}} + \|T\|\right).$$
\label{qi}
\end{lem}

\proof 
 Let $\chi_{Q_i}$ be the characteristic function of $Q_i$, and  $A_i$ be the endomorphism of $L^2(V)$ defined by $f \mapsto \chi_{Q_i} f$. Let $\Pi$ denote the projection onto the constants, so that $T \Pi = \Pi$.
The endomorphism $A_i . T . A_{i+1}$  may be decomposed:
$$A_i T A_{i+1} = A_i .\Pi . A_{i+1}  + A_i T (1-\Pi) A_{i+1}.$$ 
The endomorphism $A_i . \Pi. A_{i+1}$ may be written as $f \mapsto \frac{\chi_{Q_i}}{|V|}\langle f, \chi_{Q_{i+1}}  \rangle$, and thus has operator norm $\mu_i^{1/2} \mu_{i+1}^{1/2}$. 
Since the operator norm of $A_i . T (1-\Pi) . A_{i+1}$ is at most $\|T\|$, we conclude
that the operator norm of $A_i T A_{i+1}$ is at most $\|T\| + (\mu_i \mu_{i+1})^{1/2}$. 
The probability that a random walk visits $Q_i$ at step $i$ for every $i \in \{1,\ldots,\ell\}$ is given by
$$|V|^{-1}\langle 1_V, (A_1T A_2) (A_2 T A_3) \dots (A_{\ell-1}  T A_{\ell}) 1_V \rangle $$
and the result follows.  
  \qed

\begin{lem}   \label{wigs}
Let $Q$ be a subset of $V$, with $\mu := |Q |/ n$, and let $\gamma$ be a random walk of length $\ell$.  Then the probability that $|\gamma \cap Q| \geq (\mu + \epsilon) \ell$ is at most
$$c_1 \exp(-c_2 \ell)$$
for positive constants $c_1, c_2$ depending only on $d, \|T\|, \mu, \epsilon$. 
\end{lem}

In other words, the number of ``bad'' walks of length $\ell$, with respect to some fixed notion of ``bad'',
decays exponentially with $\ell$.

\proof

The constants $C_1, C_2, \ldots$ appearing in the proof are all understood to be positive constants depending only on $d,\|T\|,\mu,\epsilon$.

Let $S$ be a subset of $\{1..\ell\}$ of size $k$.  It follows from Lemma~\ref{qi} that the probability that $\gamma_i \in Q$ for $i \in S$ and $\gamma_i \notin Q$ for $i \notin S$ is at most 
$$   C_1 \mu^k (1-\mu)^{\ell-k} \bigl(1+\frac{\|T\|}{\min(\mu,1-\mu)}\bigr)^{\ell}.$$ 
Summing over all choices of $S$ we have that the probability that $|\gamma \cap Q| = k$ is  at most
$$   C_1 {\ell \choose k} \mu^k (1-\mu)^{\ell-k} \bigl(1+\frac{\|T\|}{\min(\mu,1-\mu)}\bigr)^{\ell} $$ 

On the other hand, the sum $$\sum_{k \geq (\mu +\epsilon) \ell} \binom{\ell}{k} \mu^k (1-\mu)^{\ell-k}$$
is at most $\exp(-C_2 \ell).$  Thus, we are done if $\|T\|$ is small enough that $(1+ \frac{\|T\|}{\min(\mu, 1-\mu)}) < e^{C_2}.$
 
 If this is not the case, we fix an integer $C_3 \geq 1$, and replace the graph
 $\mathscr{G}$ by the graph $\mathscr{G}^{(C_3)}$ for which a directed edge from $x$ to $y$ corresponds
 to a directed path of length $C_3$ in the graph $\mathscr{G}$.  (Note that $\mathscr{G}^{(C_3)}$ may have multiple edges and loops.) This improves the spectral gap:
 if $T^{(C_3)}$ is the normalized adjacency matrix of $\mathscr{G}^{(C_3)}$, we have $T^{(C_3)} = T^{C_3}$. Accordingly, $\|T^{(C_3)}\| = \|T\|^{C_3}$.  Choosing $C_3$ large enough, the argument above shows  that the probability that a random walk of length $\ell$ on the graph $\mathscr{G}^{(C_3)}$ of  remains within $Q$ for at least $(\mu + \epsilon) \ell$ steps is bounded above by  $C_4 \exp(-C_5 \ell)$.

It follows immediately that the probability that a random walk on $\mathscr{G}$
of length $C_3 \ell$ spends time $\geq (\mu + \epsilon) C_3 \ell$ inside $Q$
is at most $C_4 \exp(-C_5 \ell)$.  

This proves the desired claim.
 \qed 
\subsection{The arc graph} \label{arcgraph}

Lemma~\ref{wigs}, which tells us that an exponentially small proportion of random walks are poorly distributed in $\mathscr{G}$, will not quite suffice for our purposes; what we need to know is that an exponentially small proportion of {\em non-backtracking} walks are poorly distributed in $\mathscr{G}$.  In this section we explain how to derive such a statement from Lemma~\ref{wigs}.
 
We now assume $\mathscr{G}$ to be {\em symmetric}; that is, $E$ is closed under reversal.  For each edge $a \in E$, write $a^+$ for the target of $a$ and $a^-$ for the source of $a$.  We denote the reversal of $a$ by $\bar{a}$.
With these notations, define the arc graph $\mathscr{G}'$ to be a directed
graph whose vertices are the directed edges, or arcs, of $\mathscr{G}$.  There is an edge
from the arc $a$ to the arc $b$ exactly when $a^{+} = b^{-}$ and $a \neq \bar{b}$.

Thus, $\mathscr{G}'$ is regular of degree $d-1$.  We denote by $T'$ the normalized adjacency matrix of $\mathscr{G}'$:
$$T' F(a) = \frac{1}{d-1} \sum_\stacksum{b^{-} = a^{+}}{b \neq \bar{a} }F(b)$$
The key feature of the arc graph $\mathscr{G}'$, for us, is that we have a natural bijection
between non-backtracking paths of length $\ell $ on $\mathscr{G}$,
and paths of length $\ell-1$ on $\mathscr{G}'$. 

\subsection{}

Our goal will be to deduce a spectral gap for $T'$ from that for $T$. This is a simple analogue
of Atkin-Lehner theory in the subject of modular forms: when $d = p+1$ for some prime $p$, we can think of $\mathscr{G}$ as a quotient of the Bruhat-Tits tree attached to $\PGL_2(\Q_p)/\PGL_2(\Z_p)$; then the arc graph of the Bruhat-Tits tree is obtained by replacing $\PGL_2(\Z_p)$ with an Iwahori subgroup, so that the passage from graph to arc graph is much the same as the passage from the congruence subgroup $\Gamma_0(N)$ to the smaller subgroup $\Gamma_0(Np)$.

There are natural maps $B, E: L^2(\mathscr{G}) \rightarrow L^2( \mathscr{G}')$ (``beginning'' and ``end'') defined via
$$Bf(a) = f(a^{-}), \ Ef(a)= f(a^{+}).$$

Moreover, the orthogonal complement to $\im(B) \oplus \im(E)$ consists of those functions
$F \in L^2( \mathscr{G}')$ with the property that $$\sum_{a^{-} = v} F(a)  = \sum_{a^{+} = v} F(a)
 = 0,$$ for all $v \in V\mathscr{G}$. On this orthogonal complement (the ``new space") , the operator
 $T'$  acts via \begin{equation} F \mapsto - \frac{1}{d-1} \bar{F} \label{tnewspace} ,\end{equation}
 where $\bar{F}(a) = F(\bar{a})$. 
 Moreover, one checks  that $$\langle B f_1, E f_2 \rangle = d \langle T f_1, f_2 \rangle, \ \  T' \circ B = E.$$

Thus, if $w \in L^2(\mathscr{G})$ is an eigenfunction for $T$ with eigenvalue $\lambda$, 
then the ``old space" $ \C (Bw) + \C (Ew)$ is stable under $T'$. 
From this we see that every eigenvalue of $T'$ on this space
 is also an eigenvalue of the matrix :
$$\left( \begin{array}{cc} 0 & 1 \\ \frac{-1}{(d-1)}   & \frac{ d  \lambda}{(d-1)} \end{array} \right)$$
It is easily computed that the eigenvalues are bounded away from $1$ if $\lambda$ is. 
 By \eqref{tnewspace}, the eigenvalues of $T'$ on the new space are bounded in absolute value by $1/(d-1) < 1$.  We conclude that  $\|T'\|$ is bounded away from $1$ if $\|T\|$ is. 

\begin{prop} \label{chernofflemma}
Let $\mathscr{G} = (V,E)$ be an undirected graph with $\|T\| < 1$,  and let $Q$ be a subset of $V$, with $\mu := |Q| / n$.  The probability that a random walk {\em without backtracking} of length $\ell$
spends more than $(\mu+\epsilon) \ell$ time in $Q$ is at most
$$c_1 \exp(-c_2 \ell)$$
for constants $c_1, c_2 > 0$ depending only on $d, \|T\|, \mu, \epsilon$. 
\end{prop}
\proof
Let $Q' \subset V \mathscr{G}'$ be the subset of arcs whose initial vertex lies inside $Q$.
Noting that $\frac{|Q'|}{|\mathscr{G}'|} = \frac{|Q|} {|V|}$, we apply Lemma \ref{wigs}
 to $(\mathscr{G}', Q')$ taking into account that $\|T'\|$ is bounded away from $1$ in terms of $\|T\|$.   \qed

\part{Related topics}
\section{Miscellanea and extensions of the ergodic method.}
 \label{Extensions}
In this section we comment on various modifications that can be made in the assumptions or the proofs.

\subsection{The squarefreeness assumption} Through this paper, we have assumed that $d$ is {\em squarefree}. This is mainly for simplifying the exposition. It is  true in general that some quotient of $\Hd$ by a subgroup of $\SO_{3}(\Zz)$ is acted on by the ideal class group of some quadratic order
 containing $\sqrt{-d}$.  However, when $d$ is not squarefree this action is not necessarily homogeneous: there may be {\em several orbits}.
  This has to do with the existence of {\em non-primitive elements in $\Hd$}.  An element $\bfx=(a,b,c)\in\Hd$
   is primitive iff the integers $a,b,c$ are coprime; when $d$ is squarefree, every element of $\Hd$ is primitive. On the other hand, if we choose $d$ of the form $d_0 f^2$, the subset $f^2\H_{d_0} \subset \Hd$ forms a separate orbit of the class group action. 
The appropriate repair is to consider just the set of primitive elements $\Hd^{\mathrm{p}}\subset\Hd$; these elements again form a homogeneous space for a certain class group.  By the methods of this paper one can then show
    the equidistribution of $\rmr_{q}(\Hd^{\mathrm{p}})$ or $\rmr_{\infty}(\Hd^{\mathrm{p}})$ as $d\ra\infty$. Using the decomposition
    $$\Hd=\bigsqcup_{f^2|d}f.\mathscr{H}^{\mathrm{p}}_{d/f^2}$$ one can conclude that $\rmr_{q}(\Hd)$ and $\rmr_{\infty}(\Hd)$
    are indeed equidistributed.
\subsection{Distribution of trajectories}  One of the important ingredients for the proof of Theorem \ref{mvQ} has been to consider the {\em trajectories} associated to the various points in $\Hd$ and to show that these trajectories are, in a suitable sense, ``well-spaced''. 

The principle of lifting a point to its trajectory and then of studying the distribution of the trajectories occurs in several recent applications of ergodic methods to number theory.  For instance, a (much less evident) version of this principle occurs in the proof of the arithmetic quantum unique ergodicity conjecture, in which the role of the lifted trajectories is played by the ``microlocal lifts'' \cite{Lin}.

In fact it is possible to extrapolate the above principle to showing that {\em segments of trajectories of bounded length} become equidistributed. For instance, in the context of  Theorem \ref{linnikQ}, one can prove the following result:
 \begin{thm*} Under the assumptions of Theorem \ref{linnikQ}, fix $\ell_{0}>0$, and let $\gamma_{0}^{(\ell_{0})}$ be a fixed
  non-backtracking marked path on $\Hd(q)$ of length $2\ell_{0}+1$; then the cardinality of the set of $\x\in\Hd$ for which
  $$\gamma_{\x}^{(\ell_{0})}=\gamma_{0}^{(\ell_{0})}$$
  is equal to
  $$\frac{1}{3 \cdot 5^{2\ell_0-1}|\Hd(q)|}|\Hd|(1+o(1)),\ d\ra\infty.$$
  \label{5adic}
 \end{thm*}
 
(Note that $3 \cdot 5^{2\ell_0-1}|\Hd(q)|$ is the total number of non-backtracking paths on $\H_d(q)$ of length $2\ell_0 + 1$.)

 In combination with Proposition~\ref{Naddles} this implies that $\Hd$ is well-distributed $5$-adically, as well as mod $q$.  This may be seen as the analogue of Duke's theorem on the equidistribution of closed geodesics on the {\em unit tangent bundle of} the modular surface $X_{0}(1)$ (cf. \cite{Duke}[Theorem 1]).  See \cite{ELMV2} for a purely dynamical proof of the latter, following principles similar to the ones of the present paper (although in a different language). This theorem also admits a variant in the context of Theorem \ref{linnikinfty} and uniform versions analogous to Theorems \ref{MVinfty} and \ref{mvQ}.

\subsection{Application to mixing} 
As an application of the uniformity reached in Theorems \ref{linnikinfty} and \ref{mvQ}, we have the following mixing type results:
\begin{cor}
Fix $\delta > 0$.   Let $q$
 be a fixed integer coprime with $30$. For each squarefree $d\equiv\pm 1(5)$
 and coprime with $q$, let $\mathfrak{p}_d$ be a primitive\footnote{A primitive ideal is an integral ideal of minimal norm in its ideal class.} ideal of $\OO_{K}$ such that $\Norm(\mathfrak{p}_d) < d^{1/2-\delta}$; we suppose moreover that as $d\ra+\infty$, $\Norm(\mathfrak{p}_d) \rightarrow +\infty$. Then, for $d\ra\infty$
the set \begin{equation} \label{set1} \{\rmr_{\infty}(\tilde\bfx, [\mathfrak{p}_d]. \tilde\bfx),\  \tilde\bfx \in \SO_3(\Z)\bash\Hd\} \subset \SO_3(\Z)\bash S^2 \times \SO_3(\Z)\bash S^2 \ \end{equation} 
becomes equidistributed on $\SO_3(\Z)\bash S^2 \times \SO_3(\Z)\bash S^2$ with respect to the product of Lebesgue measures on each factor.
Similarly, the  multiset \begin{equation} \label{set2} \{\rmr_{q}(\tilde\bfx, [\mathfrak{p}_d]. \tilde\bfx),\  \tilde\bfx \in \SO_3(\Z)\bash\Hd\} \subset \SO_3(\Z)\bash \Hd(q) \times \SO_3(\Z)\bash \Hd(q) \ \end{equation}
  becomes equidistributed on $\SO_3(\Z)\bash \Hd(q) \times \SO_3(\Z)\bash \Hd(q)$ with respect to the counting measure.
\end{cor}
 Notice that if the norm $\Norm(\mfp_{d})$ remains constant for $d$ varying (possibly over a suitable subsequence), a simple variant of Theorem \ref{MVinfty} above shows that the set \eqref{set1} (resp. \eqref{set2}) becomes equidistributed along some codimension two subvariety of $(\SO_3(\Z)\bash S^2)^2$ (resp. a multiset of order $\approx \Norm(\mathfrak{p}_d)q^2$ in $(\SO_3(\Z)\bash \Hd(q))^2$\ ), namely the graph of the $\Norm(\mfp_{d})$-th Hecke correspondance. This explains the condition $\Norm(\mathfrak{p}_d) \rightarrow \infty$.
 
 In \cite{MVNV}, the second and third authors. proposed a ``mixing conjecture'': namely the quoted statement should remain true
so long as $\Norm(\mathfrak{p}_d)$ 
tends to $\infty$ as $d \rightarrow \infty$. Not surprisingly, the range {\em not} covered by the corollary (i.e. $d^{1/2-\eps(d)}\leq \Norm(\mathfrak{p}_d)\ll d^{1/2}\log d $ for $0<\eps(d)\ra 0$) is by far the most interesting range of the conjecture; for instance, the applications sketched in \cite{MVNV} depend crucially on this range.

\subsection{Extensions to other quadratic forms.}
For the {\em indefinite} discriminant quadratic form $x^2-yz$, results analogous to Linnik's theorems have been obtained by Skubenko  using the ergodic method (see \cite[Chap VI]{EPNF} and the references inside). In \cite{ELMV2}, the second and third authors together with M. Einsiedler and E. Lindenstrauss gave an alternative exposition of the ergodic approach which in fact allows for stronger results.

More generally, Linnik's ergodic method carries over when one replaces the quadratic form $x^2+y^2+z^2$ defined over $\Qq$ by any non-degenerate ternary quadratic form $Q$ defined over a fixed number field $F$. For instance, suitable generalizations of Theorems \ref{linnikinfty} and \ref{linnikQ} over number fields have been obtained by Teterin \cite{Teterin}. All the main theorems of the present text remain valid when suitably adapted to this more general context.  In this setting, the condition $d \equiv \pm 1$ modulo $5$ is 
to be replaced by the existence of a fixed prime ideal in $F$ which {\em splits} in $F(\sqrt{-d \, \mathrm{disc}(Q)})$. 
\subsection{The discriminant aspect} Theorems \ref{mvQ} and \ref{MVinfty} give results that have nearly optimal uniformity in the modulus $q$ or in the radius of the ball
  $\rho$.  However, the ternary form $Q$ whose representations of integers we study is fixed as $Q(x,y,z) = x^2 + y^2 + z^2$.  It would be interesting to investigate the option
 of proving analogues of these results where $Q$ is allowed to vary, with the aim of obtaining optimal uniformity in the discrimant
 $\mathrm{disc}(Q)$. In particular, if the form is definite,  its genus grows on the order of $\disc(Q)$.
One would envisage an analogue of Theorem \ref{mvQ} that establishes that {\em almost every} 
$Q' \in \mathrm{genus}(Q)$ represents $d$ (if everywhere locally representable), so long as $|\mathrm{genus}(Q)|\ll d^{1/2-\delta}$.

One example where this can be carried out completely is when $Q$ is the reduced norm on the space of trace zero elements in the quaternion algebra which  is
ramified at $\infty$ and at some prime $q>2$: then $|\gen(Q)|\sim \frac{q}{12}$. In that case the equidistribution
results  may be interpreted in terms of supersingular reduction of CM elliptic curves (cf. the paper of D. Gross\cite{gross}): 

Let $\mcE^{ss}_{q}$ denote the set of isomorphism classes
 of supersingular elliptic curves over $\overline{\Ff}_q$; these are in fact defined over $\Ff_{q^2}$. Then $\mcE^{ss}_{q}$ is canonically identified with $\gen(Q)$
 and so has size $\sim q/12$. For $K$ an imaginary quadratic field in which $q$ is inert, let $\mcE_{\OO_{K}}$ denote the set of elliptic curves in characteristic $0$ with
 complex multiplication by $\OO_{K}$; these curves are defined over the Hilbert class field of $K$, and 
 $$|\mcE_{\OO_{K}}|=|\disc(K)|^{1/2+o(1)}.$$ For
 $\mfq$, a place in $\ov\Qq$ above $p$, we have a ``reduction mod $\mfq$''-map
 $$\rmr_{\mfq}:\mcE_{\OO_{K}}\mapsto \mcE^{ss}_{q}.$$
 For any $\ov E\in \mcE^{ss}_{q}$ define
 $$\dev(\ov E)=\frac{1}{\mu(\ov E)}\frac{|\rmr_{\mfq}^{-1}(\ov E)|}{|\mcE_{\OO_{K}}|}-1$$
 where
 $$\mu(\ov E)=\frac{1/|\Aut(\ov E)|}{\sum_{{\ov E'}\in\mcE^{ss}_{q}}1/|\Aut(\ov E')|};$$
$\mu$ defines a probability measure on $\mcE^{ss}_{q}$.
 
 An analog of Theorem \ref{mvQ} in this context is the following:
 \begin{thm}[Equidistribution of ``Grauss'' points] Fix $\delta , \eta > 0$. Let $q$ be as above, and let $p$,  be a fixed auxiliary prime with $(p,2q)=1$. Let $K$ range over the imaginary quadratic fields such that
 \begin{enumerate}
 \item $q$ is inert in $K$,
 \item $p$ is split in $K$.\label{split}
 \end{enumerate} 
Set $d=|\disc(K)|$. As long as $q\leq d^{1/2-\delta}$, the $\mu$-measure\footnote{one could replace ``the $\mu$-measure of'' by ``the fraction of'' since $1\leq \Aut(\ov E)\leq 12$}
 of $\ov E\in \mcE^{ss}_{q}$ such that 
 $$|\dev(\ov E)|>\eta$$
 tends to $0$ as $d\ra+\infty$.
 \end{thm}
 The proof is similar to that of Theorem \ref{mvQ}; what is needed is the analog of Corollary \ref{meansquare}. It is proven by a suitable
generalization of the geometric part of the Gross formulas \cite[Prop. 10.8]{gross}. As noted in \cite{Michel}, using the formulas of Gross
 and subconvex bounds for $L$-functions, one can prove a version of this result, with $\mu$-measure $0$ and without the condition \refs{split},
as long as $q$ is less than a small positive power of $d$; under GRH this may be further extended to $q\leq d^{1/4-\delta}$.

\subsection{ Representation of quadratic forms by quadratic forms}

As we briefly discussed in \S \ref{linniklemma} the problem of studying the representation of an integer by a quadratic form admits another sort of generalization.   Namely, one can ask about representations of a varying integral quadratic form of rank $m$ by another, fixed, quadratic form of higher rank $n\geq m$.  When $m=1$, $n=3$, and the rank-$3$ form is $x^2 + y^2 + z^2$, we are in the situation of the present paper.   The general problem basically splits into two questions: 
\begin{itemize}
\item[-] the existence of such representations;
\item[-] how these are distributed (when they exist) in real or $p$-adic topologies. 
\end{itemize} 

As is explained in \cite{EV}, both the existence and the distribution problem can be understood in terms of the distribution properties of adelic orbits of some algebraic subgroup $H$ inside an adelic  homogeneous space associated to some ambient algebraic group $G \supset H$. More precisely, we can take $G$ to be a form of $\SO_n$ and $H$ a form of $\SO_{n-m}$.

The work \cite{EV} considered cases for which $n-m\geq 3$. In that case, since noncompact forms of $\SO_{n-m}$ have plenty of unipotents, one can apply the results of M. Ratner to show that (under suitable technical hypotheses) the adelic orbit of the smaller group equidistributes in the adelic quotient of the larger group; this suffices to prove the existence of such representations.   The second question, regarding equidistribution, was not addressed in \cite{EV} but likely follows from similar methods. 

This is in sharp contrast with the present situation where $n=3$, $m=1$. In this case, the smaller group is a torus (a form of $\SO_{2}$); in particular, Ratner's theorems do not apply, and one needs to use the special methods described in the present paper. It should be noted, moreover, that we are helped by the fact that $\SO_2 \subset \SO_3$ is a  {\em maximal} torus; it seems that the problem of studying representations of, say, forms of rank $2$ by a given form of rank $4$, would be significantly harder. 
\section{Harmonic analysis and Linnik's method.} \label{harmonic}

\subsection{Duke's approach via harmonic analysis} \label{dukeHA}
Although our main concern is with extending Linnik's method, we briefly remark on Duke's method and the relevance of modular $L$-functions to our problem. 

As we have remarked, the condition ``$d \equiv \pm 1$ modulo $5$'' has the effect
of ensuring that $5$ splits in $\Q(\sqrt{-d})$; one could replace the role of $5$ by any other fixed prime, as did Linnik. The problem of removing such a constraint, however, proved difficult. This question, i.e.:
\begin{equation}\label{ls}
\mbox{Can we guarantee the existence of a small $p$ with $\left(\frac{-d}{p} \right) = 1$?}\end{equation}
is a very difficult and deep one, for it is intimately connected to the issue of the Landau-Siegel ``exceptional'' zero, i.e. to effective lower bounds
 for the class number of $\Q(\sqrt{-d})$.  For instance, observe that proving that there exists a prime $p$ as in \eqref{ls}
with $p \leq d^{\epsilon}$ would immediately show that
the class number of $\Q(\sqrt{-d})$ tends to $\infty$ as $d\ra \infty$; itself, far from a trivial result.  
 
It was therefore a considerable breakthrough when W. Duke\footnote{partly in collaboration with Schulze-Pillot} \cite{Duke,DSP} established Theorem \ref{linnikQ} and Theorem \ref{linnikinfty}
without the auxiliary assumption. Duke's original approach used the Maass-Shimura theta correspondence to
express the Weyl sums attached to these equidistribution problems, i.e.
$$W(\varphi;d)=\frac{1}{|\Hd|}\sum_{\x\in\Hd}\varphi (\obfx)$$
for $\varphi$ is a (smooth) function on $S^2$ (resp. $\Hd(q)$) with zero mean with respect to the Lebesgue (resp. the counting) measure, in terms of Fourier coefficients of some half-integral weight holomorphic form $\theta(\varphi)$ \cite{Waldspurger1}; indeed, $\theta(\varphi)$
is a classical $\theta$-series. 
The decay of the Weyl sums then followed from non-trivial bounds for these Fourier coefficients, which were obtained using 
in a key way ideas and results of H. Iwaniec \cite{Iw87}.

\subsection{Central values of automorphic $L$-function: Waldspurger's formula}
It was quickly realized that these problems are also related to $L$-functions and to the problem of bounding them non-trivially at the central point: if $\varphi$ is an eigenfunction of (almost all) the Hecke operators\footnote{This is not a restriction, since the space of functions on $S^2$ or $\Hd(q)$ is generated by such Hecke eigenforms.}, acting on $S^2$ (resp. $\Hd(q)$), then by a formula of Waldspurger \cite{Waldspurger2} one has
\be\label{waldspurger}\frac{|W(\varphi;d)|^2}{\peter{\varphi,\varphi}}=c\frac{d^{1/2}}{|\Hd|^2}\frac{L(\pi,1/2)L(\pi\otimes\chi_{d},1/2)}{L(\pi,\Ad,1/2)}I_{\infty}(\varphi;d)\prod_{p|6q}I_{p}(\varphi;d)\ee
where  
\begin{itemize}
\item[-] $\peter{\ ,\ }$ is the inner product associated to the Lebesgue (resp. counting) probability measure on $S^2$ (resp. $\Hd(q)$)  and $c>0$ is an absolute explicit constant;
\item[-] $\pi$ is the automorphic representation of $\PGL_{2}$ in Jacquet-Langlands correspondence with the automorphic representation of $\SO_{3}\cong\PB$ generated by $\varphi$; $\pi$ is a discrete series representation (i.e. corresponding to an holomorphic modular form) of conductor dividing $18q^2$. \item[-] $\chi_{d}$ is the Kronecker symbol associated with the quadratic field
 $\Q(\sqrt{-d})$;
 \item[-] the factors $I_{\infty}(\varphi;d)$, $I_{p}(\varphi;d)$, $p|6q$ are some local integrals (real or $p$-adic) which can be bounded explicitly in terms of $\varphi$ but {\em independently} of $d$.
 \end{itemize} 
 From this formula, one sees that the decay of $W(\varphi;d)$ as $d\ra\infty$ is a consequence of the so-called {\em subconvex} bound: there is some absolute $\delta>0$
 \be\label{subconvex}L(\pi\otimes\chi_{d},1/2)\ll_{\pi}d^{1/2-\delta}.\ee
 Such a bound was proven, for the first time, in \cite{Iw87}.
   A few years later, a more general bound (with $1/2$ replaced by $1/2 + it $) was obtained by Duke, Friedlander, Iwaniec by an entirely different method \cite{DFI1} (again without any condition on $d$). 
 \begin{rem*} The method of Duke has been generalized by Duke and Schulze-Pillot to general ternary quadratic forms over $\Qq$~\cite{DSP}. More recently, the alternative approach involving Waldspurger's formula \refs{waldspurger} has been extended to treat ternary forms over general number fields, in particular, for totally definite ternary forms over totally real number fields: this has been carried out by Cogdell, Piatetsky-Shapiro and Sarnak \cite{CPSS}; in this generality the required subconvex bound may be found in \cite{Ve}.  For more comments about this we refer to \cite{MVICM} and to \cite{BHGAFA} where a complete exposition of these arguments is given along with sharp form of \refs{subconvex}.
 \end{rem*}

\subsection{Ergodic Theory vs. Harmonic Analysis}
Now let us compare what the ergodic and harmonic-analysis techniques yield,
as regards to Theorem \ref{mvQ} and \ref{MVinfty}. It is possible to show that the product of the local integrals $I_{\infty}(\varphi;d)\prod_{p|6q}I_{p}(\varphi;d)$ is bounded in terms of some Sobolev norm on $\varphi$. Moreover, the dependence on $\pi$ in the bound \refs{subconvex}
is polynomial in $Q(\pi)$, the analytic conductor of $\pi$ (as defined by Iwaniec and Sarnak \cite{IS}). From this one can deduce, without any condition on $d$ that,  for {\em every} $\obfx$ (resp. every $x\in S^2$), $|\dev(\obfx)|$ (resp. $|\dev(\Omega(x,\rho))|$)
is  bounded by a small negative power of $d$ as long as $q$ (resp. $\rho^{-1}$) is less than a
 (small) positive power of $d$ \cite{Blomer,Duke2,fomenko}. As we explain below, assuming the {\em Generalized Riemann Hypothesis} (GRH), this result holds when $q$ (resp. $\rho^{-1}$) is less than $\leq d^{1/8-\nu}$ for any fixed $\nu>0$.
 
We consider for simplicity the case of  $\Hd(q)$.  Let $\mcB$ denote an orthogonal basis in the space of functions on $\Hd(q)$. The average
$$\sum_{\varphi\in\mcB}\frac{|\sum_{\x\in\Hd}\varphi(\obfx)|^2}{\peter{\varphi,\varphi}}$$
 is independent of $\mcB$.  Therefore, (picking as a basis $\{\delta_{\obfx},\ \obfx\in\Hd(q)\}$) such a sum equals
$$|\Hd(q)|\sum_{\obfx\in\Hd(q)}\bigl|\rmr_{q}^{-1}(\obfx)|^2.$$ 
Taking $\mcB$ of the form $\{1\}\cup\mcB_{0}$ where $\mcB_{0}$ is constituted of Hecke eigenforms, we obtain
 \begin{align*}\sum_{\varphi\in\mcB_{0}}\frac{|\sum_{\x\in\Hd}\varphi(\obfx)|^2}{\peter{\varphi,\varphi}}&=|\Hd(q)|\sum_{\obfx\in\Hd(q)}\bigl(\rmr_{q}^{-1}(\obfx)-\frac{|\Hd|}{|\Hd(q)|}\bigr)^2\\
& =\frac{|\Hd|^2}{|\Hd(q)|}\sum_{\obfx}\dev(\obfx)^2.\end{align*}
Hence by \refs{waldspurger},
$$\frac{|\Hd|^2}{|\Hd(q)|}\sum_{\obfx}\dev(\obfx)^2=c{d^{1/2}}\sum_{\varphi\in\mcB_{0}}
\frac{L(\pi,1/2)L(\pi\otimes\chi_{d},1/2)}{L(\pi,\Ad,1/2)}I_{\infty}(\varphi;d)\prod_{p|6q}I_{p}(\varphi;d).$$
The Generalized Riemann Hypothesis
  and a local computation\footnote{to bound the local integrals $I_{p}(\varphi;d)$}, shows that
 for any  $\eps>0$
\begin{equation} \label{GRHDEV} \sum_{\obfx}\dev(\obfx)^2\ll_{\eps} (qd)^\eps q^4/d^{1/2}.\end{equation} 
Linnik's basic Lemma alone shows the same result with an additional factor $(1+d/q^4)$
on the right-hand side; in particular, \eqref{GRHDEV} is known unconditionally for $q^4 > d$. 
Now \eqref{GRHDEV} shows, for $\nu>0$, that:
\begin{itemize}
\item[-]
$\dev(\obfx)\ll d^{-\nu}$ for any $\obfx\in\Hd(q)$ as long as $q\leq d^{1/8-\nu}$
 (thus proving the surjectivity of the reduction map $\rmr_{q}$ in this range);
 \item[-]
  given $\delta>0$, the fraction of $\obfx\in\Hd(q)$ for which
$|\dev(\obfx)|>\delta$ is bounded by $\ll_{\delta} d^{-4\nu}$ as long as $q\leq d^{1/4-\nu-\epsilon}$. 
\end{itemize}

In particular we recover Theorem \ref{mvQ}. 
On the other hand we don't see at the moment, even under GRH, how to prove\footnote{This type of situation is not uncommon; for instance, it is known that the GRH implies
 that the smallest prime in an arithmetic progression with modulus $q$ is $\leq c(\varepsilon) q^{2 +\varepsilon}$; while we surely expect the result with $2 +\varepsilon$ replaced by $1 +\varepsilon$, no proof of this is known even under GRH. }
 that there are {\em no} $\obfx$  for which $\rmr^{-1}_{q}(\obfx)=\emptyset$ in the range $d^{1/8}\leq q\leq d^{1/4-\nu}$.

\begin{rem} It should be noted that \refs{GRHDEV}, obtained by pointwise application of GRH, is essentially a ``sharp on average''
bound for the family $\mathcal{F}$ of $L$-values $$\{L(\pi,1/2)L(\pi\otimes\chi_{d},1/2),\ \pi\hbox{ of conductor $18 q^2$ and trivial central character} \}.$$
The size $|\mathcal{F}|$ of that family is $\approx q^2$ while the conductor $C$ of any of these $L$-functions is of size $\approx q^4d^2$. In particular, in the range $q^4 \approx d$, we have $|\mathcal{F}| \approx |C|^{1/6}$. 
 Starting with the original work of Weyl, there are, by now, several examples of families of $L$-functions satisfying a similar relation between the size of the family and the size of the conductor and for which the central values
could be bounded sharply on average without any hypothesis \cite{Weyl,CI,XLi}.  
\end{rem}

\end{document}